\documentclass[reqno,11pt]{amsart}
\usepackage[margin=1.5in]{geometry}
\usepackage{lineno}
\usepackage{amsmath,amssymb,amsthm}
\usepackage{caption}
\usepackage{subcaption}
\usepackage{tikz}

\usepackage{ascmac,bm,framed,fancybox,mathrsfs}
\usepackage{graphicx,xcolor}
\usepackage[capitalize]{cleveref}

\newtheorem{thm}{Theorem}[section] \crefname{theorem}{theorem}{theorem}
\newtheorem{prop}[thm]{Proposition} \crefname{proposition}{proposition}{proposition}
\newtheorem{lem}[thm]{Lemma} \crefname{lemma}{lemma}{lemma}
\newtheorem{cor}[thm]{Corollary} \crefname{corollary}{corollary}{corollary}

\theoremstyle{definition}
\newtheorem{defi}[thm]{Definition} \crefname{definition}{definition}{definition}
\newtheorem{ex}[thm]{Example}

\theoremstyle{remark}
\newtheorem{rem}[thm]{Remark} \crefname{remark}{remark}{remark}

\crefname{equation}{the equation}{the equation}
\crefname{figure}{figure}{figure}

\numberwithin{equation}{section}

\newcommand{\relmiddle}[1]{\mathrel{}\middle#1\mathrel{}}

\newcommand{\Hx}{{\rm (H1)}}
\newcommand{\Hp}{{\rm (H2)}}
\newcommand{\Hu}{{\rm (H3)}}
\newcommand{\Hc}{{\rm (H4)}}
\newcommand{\Hsc}{{\rm (H4)$_{\mathrm{st}}$}}
\newcommand{\Hr}{{\rm (H5)}}
\newcommand{\Uo}{{\rm (U)}}

\begin{document}

\title[Lower gradient estimates for viscosity solutions]{%
Lower gradient estimates for viscosity solutions to first-order Hamilton--Jacobi equations depending on the unknown function}
\author{Kazuya Hirose}
\address[K.~Hirose]{%
Department of Mathematics, Graduate School of Science, Hokkaido University,
Kita 10, Nishi 8, Kita-Ku, Sapporo, Hokkaido, 060-0810, Japan}
\email{hirose.kazuya.w2@elms.hokudai.ac.jp}
\date{\today}




\begin{abstract}
In this paper, we derive the lower bounds for the gradients of viscosity solutions to the Hamilton--Jacobi equation, where the convex Hamiltonian depends on the unknown function. We obtain gradient estimates using two different methods. 
First, we utilize the equivalence between viscosity solutions and Barron--Jensen solutions to study the properties of the inf-convolution. 
Second, we examine the Lie equation to understand how initial gradients propagate along its solutions.
\end{abstract}

\maketitle

\section{Introduction}
\subsection{Equation and Purpose}
In this paper,
we consider the first-order Hamilton--Jacobi equation
\begin{equation}
u_t(x,t)+H(x,t,u(x,t),D_x u(x,t))=0  
\quad \mathrm{in} \ \mathbb{R}^n \times (0,T)
\label{eq:HJu}
\end{equation}
with the initial condition
\begin{equation}
u(x,0)=u_0(x)
\quad \mathrm{in} \ \mathbb{R}^n.
\label{eq:ID}
\end{equation}
Here $u: \mathbb{R}^n \times [0,T)  \to \mathbb{R}$ is the unknown function, the Hamiltonian $H: \mathbb{R}^n \times [0,T] \times \mathbb{R} \times \mathbb{R}^n \to \mathbb{R}$ is continuous, and the initial datum $u_0: \mathbb{R}^n \to \mathbb{R}$ is Lipschitz continuous.
We denote $u_t=\partial_t u$ and $D_x u=(\partial_{x_1}u,\ldots,\partial_{x_n}u)$.

There is abundant literature on the upper bounds for the gradients of viscosity solutions, which have been applied in various contexts, such as the existence and regularity of solutions to nonlinear elliptic and parabolic equations and singular perturbation problems (for example, see \cite{B.90}). 
Upper gradient estimates can be derived using approaches such as the weak Bernstein method for viscosity solutions \cite{B.91, B.21} or the Ishii--Lions method \cite{IL.90}.
For more information, see \cite{LN.16} and references therein.

On the other hand, however, there has been little research on the lower bounds for gradients.
One of the reasons is that the aforementioned methods cannot be applied to derive the lower bounds.
Additionally, as mentioned in \cite[Example 7.5]{HH.23}, lower gradient estimates do not generally hold when the Hamiltonian $H=H(x,t,u,p)$ is not convex in the variable $p$. 
Therefore, it is crucial to consider how to use the convex structure of the Hamiltonian $H$ to derive the lower bounds for gradients.

When the Hamiltonian is of the form $H=H(x,t,p)$, the lower bounds for the gradients of viscosity solutions have already been studied in \cite{L.01, HH.23}.
However, to the best of the author's knowledge, there are no results on lower gradient estimates for viscosity solutions to \eqref{eq:HJu}--\eqref{eq:ID}.

This paper aims to derive the lower bounds for the spatial derivatives of viscosity solutions to \eqref{eq:HJu}--\eqref{eq:ID}
with two different ways.
The first is the approximation method by \cite{L.01}, and the second is the dynamical method by \cite{HH.23}.
Furthermore, we will compare these results and determine which method is superior.

\subsection{Assumptions and Known results}

The assumptions on $H$ and $u_0$ are the following:
\begin{itemize}
\item[{\Hx}]
There exist $C_1 \geqq 0$ and $\beta \in \{ 0, 1 \}$ such that 
\[
| H (x,t,u,p)-H (y,t,u,p) |
\leqq C_1 (\beta+|p|) |x-y|
\]
for all $(t,u,p) \in [0,T] \times \mathbb{R} \times \mathbb{R}^n$ and $x,y \in \mathbb{R}^n$.

\item[{\Hp}]
There exist $A_2,B_2 \geqq 0$ such that 
\[
| H (x,t,u,p)-H (x,t,u,q) |
\leqq (A_2 |x| +B_2)|p-q|
\]
for all $(x,t,u) \in \mathbb{R}^n \times [0,T] \times \mathbb{R}$ and $p,q \in \mathbb{R}^n$.

\item[{\Hu}]
There exists $K_3 \geqq 0$ such that
\[
|H(x,t,u,p)-H(x,t,v,p)|
\leqq K_3|u-v|
\]
for all $(x,t,p) \in \mathbb{R}^n \times [0,T] \times \mathbb{R}^n$ and $u,v \in \mathbb{R}$.

\item[{\Hc}]
The function $p \mapsto H(x,t,u,p)$ is convex in $\mathbb{R}^n$ for all $(x,t,u) \in \mathbb{R}^n \times [0,T] \times \mathbb{R}$.

\item[{\Uo}]
For $x_0 \in \mathbb{R}^n$ and $r>0$, there exists a constant $\theta>0$ such that
\[
|p| \geqq \theta \quad 
\mbox{for all $x \in B_r(x_0)$ and $p \in D^- u_0(x)$.}
\]
\end{itemize}
Here $| \cdot |$ stands for the standard Euclidean norm, $B_r(x)$ denotes the open ball with radius $r$ centered at $x$, and $D^-u_0(x)$ denotes the subdifferential of $u_0$ at $x$ (see \cref{def:subdi}).
The existence and uniqueness of solutions to \eqref{eq:HJu}--\eqref{eq:ID} are well-known under further assumptions such as uniform continuity and monotonicity in the variable $u$ (for example, \cite{ABIL.11, I.86}).
However, we do not make such assumptions in this context.

As mentioned above, for the equation 
\begin{equation}
u_t(x,t)+H(x,t,D_x u(x,t))=0
\quad \mathrm{in} \ \mathbb{R}^n \times (0,T),
\label{eq:HJ}
\end{equation}
the lower bounds for the gradients of viscosity solutions have been studied in \cite{L.01, HH.23}.
The results are as follows.

\begin{thm}[{\cite[Theorem 4.2]{L.01}}]\label{thm:leyprelb}
Assume {\Hx}, {\Hp}, {\Hc} and {\Uo}.
Let $u \in C(\mathbb{R}^n \times [0,T))$ be a viscosity solution of \eqref{eq:HJ}--\eqref{eq:ID}.
Then, for any $t_0 \in (0,T]$ satisfying
\[
\theta^2 - 2 \beta C_1 e^{\frac{5}{2}C_1 T}t_0>0,
\]
we have
\begin{equation}
|p| \geqq e^{-\frac{5}{4}C_1 t} \sqrt{\theta^2 - 2 \beta C_1 e^{\frac{5}{2}C_1 T}t_0}
\label{eq:estOley1}
\end{equation}
for all $(x,t) \in \mathcal{D}(x_0,r) \cap (\mathbb{R}^n \times (0,t_0))$ and $p \in D^-_x u(x,t)$.
Here $\beta$ and $C_1$ are the constants in {\Hx} and
\begin{equation}
\mathcal{D}(x_0,r):=
\{ (x, t) \in B_r(x_0) \times (0,T) \mid e^{(A_2+B_2+A_2|x_0|)t}(1+|x-x_0|) < r+1 \},
\label{eq:leydod}
\end{equation}
where $A_2$ and $B_2$ are the constants in {\Hp}.
\end{thm}

\begin{thm}[{\cite[Theorems 1.1, 6.2 and Proposition 6.8]{HH.23}}] \label{thm:ourprelb}
Assume {\Hx}, {\Hp} and {\Hc}.
Let $u \in C(\mathbb{R}^n \times [0,T))$ be a viscosity solution of \eqref{eq:HJ}--\eqref{eq:ID}.
Then we have
\[
\underline{I}(x,t;u_0) e^{-C_1t} -\beta (1-e^{-C_1t})
\leqq |p|
\leqq \overline{S}(x,t;u_0) e^{C_1t} +\beta (e^{C_1t}-1)
\]
for all $(x,t) \in \mathbb{R}^n \times (0,T)$ and $p \in D^-_x u(x,t)$.
Here $\beta$ and $C_1$ are the constants in {\Hx} and
\begin{align}
\overline{S}(x,t;u_0)
&:=\lim_{\delta \to +0} \sup \left\{ |p| \relmiddle| 
p \in D^-u_0(y),~y \in \overline{B_{R(x,t)+\delta}(x)} \right\}, \notag\\
\underline{I}(x,t;u_0)
&:=\lim_{\delta \to +0} \inf \left\{ |p| \relmiddle| 
p \in D^-u_0(y),~y \in \overline{B_{R(x,t)+\delta}(x)} \right\}, \notag\\
R(x,t)
&:=\begin{cases} 
\left( \dfrac{B_2}{A_2}+|x| \right)(e^{A_2t}-1) & (A_2>0), 
\smallskip \\ 
B_2 t & (A_2=0).
\end{cases} \label{eq:defiR}
\end{align}
Moreover, assume {\Uo}.
Then we have
\[
|p| \geqq 
\theta e^{-C_1t} -\beta (1-e^{-C_1t})
\]
for all $(x,t) \in \mathcal{E}(x_0,r)$ and $p \in D_x^- u(x,t)$, where
\begin{equation}
\mathcal{E}(x_0,r)
:=\{ (x,t) \in B_r(x_0) \times (0,T) \mid R(x,t)+|x-x_0| < r \}. 
\label{eq:ourdod}
\end{equation}
Furthermore, we have $\mathcal{D}(x_0,r) \subset \mathcal{E}(x_0,r)$.
\end{thm}

In \cite{L.01}, Olivier Ley derives gradient estimates using the notion of Barron--Jensen solutions, which are equivalent to viscosity solutions when $H$ is convex, and by studying the properties of the inf-convolution. 
In \cite{HH.23}, Nao Hamamuki and the author derive gradient estimates using the result in \cite{ACS.20} and by studying how the initial gradients propagate along solutions to Hamiltonian systems (for example, see \cite{CS.04}). 
Moreover, they show that the result in \cite{HH.23} is better than the one in \cite{L.01} in several senses.

\subsection{Methods and Main results}
Our main results are to derive the lower bounds for the gradients of solutions to \eqref{eq:HJu}--\eqref{eq:ID} in two different ways.

Our first main result can be obtained using the approximation method.
The idea is similar to that in \cite{L.01}.
To derive the lower bounds, we use the key lemma (\cref{cor:infconvsp}) that if $u$ is a Barron--Jensen solution of \eqref{eq:HJu}--\eqref{eq:ID}, then the inf-convolution $u_{\varepsilon}$ defined in \eqref{eq:infconvsp} is a viscosity subsolution of \eqref{eq:HJu} with the error term 
\[
u_t(x,t)+H(x,t,u(x,t),D_x u(x,t))=
\dfrac{\beta C_1}{2}e^{\gamma t}\varepsilon^2.
\]
To use this fact, we need to prove the equivalence between Barron--Jensen solutions and viscosity solutions when the Hamiltonian $H$ is convex (\cref{thm:equivvb}).
We then apply the local comparison principle (\cref{thm:unCP}) for a subsolution $u_{\varepsilon}$ and supersolution $u$, and obtain
\[
u_{\varepsilon}(x,t)-u(x,t)
\leqq \sup_{y \in \overline{B_{r}(x_0)}} e^{-K_3 t} (u_{\varepsilon}-u)(y,0)
 +e^{-K_3 t} \int_0^t e^{K_3 s}\dfrac{\beta C_1}{2}e^{\gamma s}\varepsilon^2\,ds
\]
for all $(x,t) \in \overline{\mathcal{E}(x_0,r)}$, where $\mathcal{E}(x_0,r)$ is defined in \eqref{eq:ourdod} and $\overline{\mathcal{E}(x_0,r)}$ is its closure.
The first term on the right-hand side is bounded from above under assumption {\Uo} by the viscosity decreasing principle (see \cite[Lemma 2.3]{BL.06} and \cite[Proposition 4.1]{L.01}).

On the other hand, by the definition of subdifferentials and Taylor expansion, there exist a continuous function $\omega: [0,\infty) \to [0,\infty)$ satisfying $\lim_{r \to +0} \omega(r)=0$ and a constant $M>0$ such that
\[
u_{\varepsilon}(x,t)-u(x,t)
\geqq -\dfrac{|p|^2}{4}e^{\gamma t}\varepsilon^2 -M\varepsilon^2\omega(M\varepsilon^2)
\]
for all $(x,t) \in \mathcal{E}(x_0,r)$ and $p \in D_x^-u(x,t)$.
Therefore, we can deduce the lower bound estimates of $|p|$ by combining the above two inequalities and letting $\varepsilon \to 0$.
This is our first main result.

\begin{thm}[Gradient estimate I]\label{thm:leylb}
Assume {\Hx}--{\Hc} and {\Uo}.
Let $u \in C(\mathbb{R}^n \times [0,T))$ be a viscosity solution of \eqref{eq:HJu}--\eqref{eq:ID}.
Then, for any $t_0 \in (0,T]$ satisfying
\[
\theta^2 e^{-((\frac{\beta}{2}+2)C_1+2K_3)t_0}-2C_1\beta t_0>0,
\]
we have 
\begin{equation}
|p| \geqq \sqrt{\theta^2 e^{-((\frac{\beta}{2}+2)C_1+2K_3)t}-2C_1\beta t}
\label{eq:estley1}
\end{equation}
for all $(x,t) \in \mathcal{E}(x_0,r) \times (\mathbb{R}^n \times 
 (0,t_0))$ and $p \in D^-_x u(x,t)$.
Here $\beta$ and $C_1$ are the constants in {\Hx} and $K_3$ is a constant in {\Hu}.
\end{thm}

We prove \cref{thm:leylb} in \cref{sec:ley}.
While we can also derive upper bound estimates of gradients in a similar manner, we omit this.
We note that \cref{thm:leyprelb} and \cref{thm:leylb} are different even if $K_3=0$.
We can potentially enhance the results by carefully studying the method in [32] (also see \cref{rem:leyobrem}).

Our second main result can be obtained using the dynamical method.
The idea is similar to that in \cite{HH.23}.
To estimate gradients of solutions, we consider the Lie equation (or contact Hamiltonian systems)
\begin{equation}
\begin{cases}
\xi^{\prime}(s)= D_p H(\xi(s),s,u_{\xi}(s),\eta(s)), \\
\eta^{\prime}(s)=-D_x H(\xi(s),s,u_{\xi}(s),\eta(s)) -D_uH(\xi(s),s,u_{\xi}(s),\eta(s))\eta(s), \\
u_{\xi}^{\prime}(s)= \langle \eta(s),\xi^{\prime}(s) \rangle-H(\xi(s),s,u_{\xi}(s),\eta(s)).
\end{cases}
\label{eq:Hsys.int} 
\end{equation}
Here, we temporarily assume that $H$ is smooth enough.
For a differentiable point $(x,t)$ of $u$, we consider a solution $(\xi,\eta,u_{\xi})$ of \eqref{eq:Hsys.int} with the terminal condition
\[
\xi(t)=x, \quad \eta(t)=D_x u(x,t), \quad u_{\xi}(t)=u(x,t).
\]
Then, $u$ is differentiable at $(\xi(s),s)$ ($s \in (0,t)$) along $\xi$, and we have
\begin{equation}
\eta(0) \in D^- u_0(\xi(0)).
\label{eq:e0Dpr.int}
\end{equation}
Integrating both sides of the second equation in \eqref{eq:Hsys.int} over $[\tau,t]$, applying Gronwall's lemma, and letting $\tau=0$, we can derive an estimate of the form
\[
|D_x u(x,t)|
\geqq |\eta(0)| e^{-(C_1+K_3)t}-\frac{C_1\beta}{C_1+K_3}(1-e^{-(C_1+K_3)t}),
\]
if $(C_1, K_3) \neq (0,0)$, where $C_1$ and $\beta$ are the constants in {\Hx}, and $K_3$ is a constant in {\Hu}.
Moreover, we can also derive an estimate of the form
\[
|x-\xi(0)| \leqq R(x,t)
\]
similarly to the first equation in \eqref{eq:Hsys.int}, where $R(x,t)$ is defined in \eqref{eq:defiR}.
Therefore, we can deduce gradient estimates from the above two estimates and \eqref{eq:e0Dpr.int}.

For a nonsmooth viscosity solution, the above argument is justified via approximation.
To apply the results of Herglotz' variational principle \cite{CCWY.19, CCJWY.20, CH.22}, we need to impose additional conditions on $H$, where smoothness and strict convexity, as follows:
\begin{itemize}
\item[{\Hr}]
$H \in C^2(\mathbb{R}^n \times [0,T] \times \mathbb{R} \times \mathbb{R}^n)$.

\item[{\Hsc}]
$D_{pp}H(x,t,u,p)$ is positive definite for all $(x,t,u,p) \in \mathbb{R}^n \times [0,T] \times \mathbb{R} \times \mathbb{R}^n$.
\end{itemize}
These additional conditions can be removed by approximating $H$ with $H_{\varepsilon}$ satisfying {\Hr} and {\Hsc}, and studing the equation
\begin{equation}
(u_{\varepsilon})_t(x,t)+H_{\varepsilon}(x,t,u_{\varepsilon}(x,t),D_x u_{\varepsilon}(x,t))=0 
\quad \mathrm{in} \ \mathbb{R}^n \times (0,T).
\label{eq:HJuap}
\end{equation}
By applying the results to approximate solutions $u_{\varepsilon}$ of \eqref{eq:HJuap}, we can derive gradient estimates for \eqref{eq:HJu}.
This is our second main result.

\begin{thm}[Gradient estimate II]\label{thm:ourlb}
Assume {\Hx}--{\Hc}.
Let $u \in C(\mathbb{R}^n \times [0,T))$ be a viscosity solution of \eqref{eq:HJu}--\eqref{eq:ID}.
Then 
\begin{enumerate}
\item[(i)]
if $(C_1, K_3) \neq (0,0)$,
\[
\begin{split}
&\phantom{\leqq} \underline{I}(x,t;u_0)e^{-(C_1+K_3)t}-\frac{C_1\beta}{C_1+K_3}(1-e^{-(C_1+K_3)t}) \\
&\leqq |p| \\
&\leqq \overline{S}(x,t;u_0)e^{(C_1+K_3)t}+\frac{C_1\beta}{C_1+K_3}(e^{(C_1+K_3)t}-1)
\end{split}
\]
for all $(x,t) \in \mathbb{R}^n \times (0,T)$ and $p \in D_x^- u(x,t)$.

\item[(ii)]
if $(C_1, K_3)=(0,0)$,
\[
\underline{I}(x,t;u_0)
\leqq |p|
\leqq \overline{S}(x,t;u_0)
\]
for all $(x,t) \in \mathbb{R}^n \times (0,T)$ and $p \in D_x^- u(x,t)$.
\end{enumerate}
\end{thm}
The next theorem immediately follows from \cref{thm:ourleylb} under the assumption {\Uo}.

\begin{thm}[Gradient estimate II under {\Uo}]\label{thm:ourleylb}
Assume {\Hx}--{\Hc} and {\Uo}.
Let $u \in C(\mathbb{R}^n \times [0,T))$ be a viscosity solution of \eqref{eq:HJu}--\eqref{eq:ID}.
Then 
\begin{enumerate}
\item[(i)]
if $(C_1, K_3) \neq (0,0)$,
\begin{equation}
|p| 
\geqq  \theta e^{-(C_1+K_3)t}-\frac{C_1\beta}{C_1+K_3}(1-e^{-(C_1+K_3)t})
\label{eq:ourleylb}
\end{equation}
for all $(x,t) \in \mathcal{E}(x_0,r)$ and $p \in D_x^- u(x,t)$, where $\mathcal{E}(x_0,r)$ is defined in \eqref{eq:ourdod}. 

\item[(ii)]
if $(C_1, K_3)=(0,0)$,
\[
|p| \geqq \theta
\]
for all $(x,t) \in \mathcal{E}(x_0,r)$ and $p \in D_x^- u(x,t)$.
\end{enumerate}
\end{thm}
We prove \cref{thm:ourleylb} in \cref{sec:our}.
In \cref{sec:cp}, we will discuss whether \cref{thm:leylb} or \cref{thm:ourleylb} are better.
We observe that the domain where the gradient estimates are available is the same $\mathcal{E}(x_0,r)$.
The lower bounds are the same if $(C_1, K_3)=(0,0)$, but generally not the same.
We find that the lower bound in \cref{thm:ourleylb} is larger than the one in \cref{thm:leylb} (\cref{thm:complb}).

At first glance, the dynamical method seems to be superior to the approximation method, but we believe that the approximation method can be easily applied in various situations, such as metric spaces.
Therefore, both methods are important for investigating the lower bounds for gradients.

In \cref{sec:ex}, we present some examples where the estimates are optimal in some senses.
A typical Hamiltonian we consider is of the form $H(x,t,u,p)=\lambda u+H_0(x,t,p)$ for $\lambda \in \mathbb{R}$ and $H_0: \mathbb{R}^n \times [0,T] \times \mathbb{R}^n \to \mathbb{R}$ satisfying {\Hx}, {\Hp}, and {\Hc}.
This is a special case of the modified level-set equation in \cite{H.19, BFS.24}.
In this case, we can derive another gradient estimate through direct calculation (\cref{prop:estatRu1lam}).

\subsection{Literatures}

Lower gradient estimates for other equations have also been studied. 
In \cite{BLM.12}, lower gradient estimates are obtained for viscosity solutions to second-order geometric equations, and these results are applied to prove the short-time uniqueness of solutions for geometric equations with nonlocal terms. 
The proof relies on the geometric properties of the equations and the continuous dependence of solutions. 
See \cite{BL.06, BCLM.08, BCLM.09} for related results with first-order nonlocal equations, and \cite{F.08, FM.09} for related results with second-order nonlocal equations.

Furthermore, lower bounds for the spatial Lipschitz constant are investigated in \cite{F.18, HK.23}. 
In \cite{F.18}, lower bounds with optimal exponents in terms of time are derived for linear second-order uniformly parabolic equations and Hamilton--Jacobi equations with the $p$-power Hamiltonian when the initial data is Hölder continuous. 
The proof relies on the representation formula of solutions for the equation and logarithmic Sobolev inequalities. 
More general, fully nonlinear parabolic equations, including those mentioned above, are studied in \cite{HK.23}.

The notion of Barron--Jensen solution was introduced as an extended concept of viscosity solutions in \cite{BJ.90}, which posits that the equation holds in the viscosity solution sense for test functions that touch from below. 
It is known that the notion of Barron--Jensen solutions is equivalent to the notion of viscosity solutions when the Hamiltonian $H$ is convex \cite{BJ.90, L.01}. 
Furthermore, the equivalence between Barron--Jensen solutions and viscosity solutions has also been shown with the quasiconvex Hamiltonian and applied to $L^{\infty}$-optimal control problems \cite{B.13}.
Recently, \cite{IWWY.22} shows the existence and uniqueness of extended real-value solutions to the Cauchy problem with general lower semicontinuous data using the notion of Barron--Jensen solutions (see also \cite{I.01}).

The variational principle for contact Hamiltonian systems (or Lie equation) is known in two different contexts.
One of the variational formulations is the implicit variational principle introduced in \cite{WWY.17}, which plays an important role in connecting Weak KAM and Arbuy-Mather theory to contact Hamiltonian systems.
For more details, see \cite{WWY.19v, WWY.19a, WWY.21} and references therein.
Another variational formulation is Herglotz' variational principle studied in \cite{CCWY.19, CCJWY.20}, which we use in this paper.
This formulation explicitly provides solutions for the contact Hamilton--Jacobi equation and is related to nonholonomic constraints in \cite{LTW.18}.
In \cite{HCHZ.22}, various representation formulas are studied using Herglotz' variational principle.
These representation formulas are applied to the propagation of singularities \cite{CH.22}.
For the classical Hamiltonian systems, see \cite{CS.04}.

\subsection{Organization}
The rest of this paper is organized as follows.
In \cref{sec:pre}, we provide some definitions and known results utilized in the proofs.
In \cref{sec:ley}, we investigate the properties of the inf-convolution and prove \cref{thm:leylb} in \cref{sec:ley}.
In \cref{sec:our}, we analyze the properties of solutions of the Lie equation and prove \cref{thm:ourlb}.
These results are compared in \cref{sec:cp}.
In \cref{sec:ex}, we give some examples of solutions to \eqref{eq:HJu}--\eqref{eq:ID}.
Primary results are included in the Appendix, where we derive the local comparison principle in \cref{sec:LCP} and show the equivalence between Barron--Jensen solutions and viscosity solutions in \cref{sec:equiv}.

\section{Preliminaries}\label{sec:pre}

In this section, we will give some definitions, notations, and known results, which we will use later.

\subsection{Viscosity solution and Barron--Jensen solution}\label{sec:previs}

For the basic theory of viscosity solutions, see \cite{CIL.92, ABIL.11}.

\begin{defi}[Viscosity solution]\label{def:vissol}
We say that $u \in C(\mathbb{R}^n \times [0,T))$ is a {\em viscosity subsolution} (resp.\ {\em viscosity supersolution}) of \eqref{eq:HJu}--\eqref{eq:ID} if $u(x,0) \leqq u_0(x)$ (resp.\ $u(x,0) \geqq u_0(x)$) for all $x \in \mathbb{R}^n$, and for all $\phi \in C^1(\mathbb{R}^n \times (0,T))$ such that $u-\phi$ has a local maximum (resp. minimum) at $(x,t) \in \mathbb{R}^n \times (0,T)$, we have
\[
\phi_t(x,t) + H(x,t,u(x,t),D_x \phi(x,t)) \leqq 0
\quad \mbox{(resp.\ $\geqq 0$)}.
\]
We say that $u \in C(\mathbb{R}^n \times [0,T))$ is a {\em viscosity solution} if $u$ is both a viscosity subsolution and a viscosity supersolution.
\end{defi}

\begin{defi}[Barron--Jensen solution]\label{def:BJsol}
We say that a lower semicontinuous function $u: \mathbb{R}^n \times [0,T) \to \mathbb{R}$ is a {\em Barron--Jensen solution} of \eqref{eq:HJu}--\eqref{eq:ID} if $u(x,0) \leqq u_0(x)$ (resp.\ $u(x,0) \geqq u_0(x)$) for all $x \in \mathbb{R}^n$, and for all $\phi \in C^1(\mathbb{R}^n \times (0,T))$ such that $u-\phi$ has a local minimum at $(x,t) \in \mathbb{R}^n \times (0,T)$, we have
\[
\phi_t(x,t)+H(x,t,u(x,t),D_x\phi(x,t))=0.
\]
\end{defi}

\begin{rem}\label{rem:vis}
By definition, a Barron--Jensen solution is a viscosity supersolution.
\end{rem}

If $H=H(x,t,u,p)$ is convex in $p$, then the notion of Barron--Jensen solutions is equivalent to the notion of viscosity solutions
(\cite[Theorem 2.3]{BJ.90} and \cite[Theorem 3.1]{L.01}).
Note that in \cite{BJ.90}, $H$ is is assumed to be Lipschitz continuous in the variable $t$, and in \cite{L.01}, $H$ does not depend on the variable $u$ .
We can extend this equivalence in our setting under the assumption {\Hu}.

\begin{thm}\label{thm:equivvb}
Assume {\Hx},  {\Hu} and {\Hc}.
Then $u \in C(\mathbb{R}^n \times [0,T))$ is a viscosity solution of \eqref{eq:HJu}--\eqref{eq:ID} if and only if $u$ is a Barron--Jensen solution of \eqref{eq:HJu}--\eqref{eq:ID}.
\end{thm}

We prove \cref{thm:equivvb} in \cref{sec:equiv}.
In the proof, the properties of the sup- and inf-convolution are crucial.
It is known that if $u$ is a viscosity subsolution (resp. supersolution), then the sup-convolution (resp. inf-convolution) of $u$ is a viscosity subsolution (resp. supersolution) in \eqref{eq:HJu} with an error term.
Moreover, \cite[Lemma 3.2]{L.01} states that if $u$ is a Barron--Jensen solution, then the inf-convolution of $u$ is a viscosity subsolution in the same approximate equation.
We also use this property to derive lower gradient estimates (see \cref{sec:ley}).

\begin{rem}
If the Hamiltonian $H$ is not convex, then a viscosity solution of \eqref{eq:HJu}--\eqref{eq:ID} is not necessarily a Barron--Jensen solution (see \cite[Remark 2.6]{ABI.99}).
\end{rem}

We next define several notations of generalized gradients (see \cite[Section 3]{CS.04}).

\begin{defi}\label{def:subdi}
Let $f:\mathbb{R}^n \to \mathbb{R}$ and $x \in \mathbb{R}^n$.
\begin{enumerate}
\item 
We define the {\em subgradient} $D^- f(x)$ of $f$ at $x$ as 
\[
D^-f(x)
:=\left\{ D \phi (x) \relmiddle|
\phi \in C^1(\mathbb{R}),~\mbox{$f-\phi$ attains a local minimum at $x$} \right\}.
\]
The {\em supergradient} $D^+ f(x)$ of $f$ at $x$ is defined by replacing ``a local minimum'' by ``a local maximum'' in the above definition.
Moreover, for a function $u:\mathbb{R}^n \times [0,T) \to \mathbb{R}$ and $(z,t) \in \mathbb{R}^n \times (0,T)$, we define
\[
D^{\pm}_x u(z,t)
=\{ p \mid (p,\tau) \in D^{\pm} u(z,t) \}.
\]

\item 
We define the {\em reachable gradient} $D^* f(x)$ of $f$ at $x$ as 
\[
D^*f(x)
:=\left\{ p \in \mathbb{R}^n \relmiddle|
\begin{tabular}{c}
there exists $\{ x_k \}_{k=1}^{\infty} \subset  \mathbb{R}^n \setminus \{ x \}$ such that \\
$f$ is differentiable at $x_k$ and \\
$x_k \to x$, $D f(x_k) \to p$ as $k \to \infty$
\end{tabular}
\right\}.
\]
\end{enumerate}
\end{defi}

We list some properties of generalized gradients
for a semiconcave function.
\begin{prop}[{\cite[Proposition II.4.7 (b)]{BD.97}, \cite[Proposition 3.3.4, Theorem 3.3.6]{CS.04}}]\label{prop:DpmSC}
Assume that $f$ is semiconcave 
in a nonempty convex set $K \subset \Omega$
and that $x$ is an interior point of $K$.
\begin{enumerate}
\item
If $f$ is not differentiable at $x$, 
then $D^-f(x)=\emptyset$.

\item
$D^* f(x) \neq \emptyset$.
If $f$ is differentiable at $x$, then 
$D^* f(x) = \{ Df(x) \}$.
\end{enumerate}
\end{prop}

\subsection{Herglotz' variational principle and the Lie equation}

In the subsection, we explain some results of Herglotz' variational principle and the Lie equation.
For more details, see \cite{CCWY.19, CCJWY.20, CH.22}.
To use these results, the following conditions on $H$ are usually assumed:
{\Hu}, {\Hsc}, {\Hr} and 
\begin{itemize}
\item[(i)]
There exist two superlinear functions 
$\theta_0,\theta_1: [0,\infty) \to [0,\infty)$
and a constant $c_0>0$ such that
\[
\theta_0(|p|)-c_0
\leqq H(x,t,0,p)
\leqq \theta_1(|p|)
\]
for all $(x,t,p) \in \mathbb{R}^n \times [0,T] \times \mathbb{R}^n$;

\item[(ii)]
There exist two constants $C_0,C_1$ such that
\[
\max \{ D_x H(x,t,u,p),D_p H(x,t,u,p) \} \leqq C_1+C_2H(x,t,u,p)
\]
for all $(x,t,u,p) \in \mathbb{R}^n \times [0,T] \times \mathbb{R} \times \mathbb{R}^n$.
\end{itemize}
However, these additional assumptions (i)--(ii) can be removed by using the argument in \cite[page 1421]{ACS.20} and \cite[Remark 2.3]{HH.23}.
Throughout this subsection, we only assume {\Hx}--{\Hr} and {\Hsc}. 

We define the {\em Lagrangian} 
$L:\mathbb{R}^n \times [0,T] \times \mathbb{R} \times \mathbb{R}^n \to \mathbb{R}$ associated with $H$ by
\[
L(x,t,u,q)
:=\sup_{p \in \mathbb{R}^n} \{ \langle p,q \rangle-H(x,t,u,p) \}.
\]
For $x,y \in \mathbb{R}^n$ and $t \in [0,T]$, we set
\[
\Gamma_{x,y}^t
:=\{ \xi: [0,t] \to \mathbb{R}^n \mid 
\mbox{$\xi$ is absolutely continuous in $[0,t]$, $\xi(0)=x,\xi(t)=y$} \}.
\]
For $\xi \in \Gamma_{x,y}^t$, we consider the Carath\'{e}odory equation
\begin{equation}
u_{\xi}'(s)=L(\xi(s),s,u_{\xi}(s),\xi'(s)) 
\quad (\mbox{a.e. $s \in [0,t]$})
\label{eq:caraeq}
\end{equation}
with $u_{\xi}(0)=u_0(\xi(0))$.
By the Carath\'{e}odory's theorem (see \cite[Proposition A.1]{CCJWY.20}), \eqref{eq:caraeq} has a unique solution $u_{\xi}$ if $s \mapsto L(\xi(s),s,a,\xi'(s))$ belongs to $L^1([0,t])$ for all $a \in \mathbb{R}$.

We also define the functional
\[
J(\xi):=
u_0(\xi(0))+\int_0^t L(\xi(s),s,u_{\xi}(s),\xi'(s)) \,ds
\]
and 
\[
\mathcal{A}:=
\{ \xi \in \Gamma_{x,y}^t \mid
\mbox{$s \mapsto L(\xi(s),s,a,\xi'(s))$ belongs to $L^1([0,t])$ for all $a \in \mathbb{R}$} \}.
\]

\begin{prop}[{\cite[Proposition 1]{CCWY.19}, \cite[Proposition 1.1]{CCJWY.20}}]
The functional $J(\xi)$ has a minimizer $\xi^{\ast} \in \mathcal{A}$.
\end{prop}

Let $\mathcal{A}_{\min}$ be the set of all the minimizers of $J(\xi)$ over $\mathcal{A}$.

\begin{prop}[{\cite[Theorem 1]{CCWY.19}, \cite[Proposition 2.1(a)]{CCJWY.20}}]
Let $\xi \in \mathcal{A}_{min}$ and $u_{\xi}$ be a solution of \eqref{eq:caraeq} with $u_{\xi}(0)=u_0(\xi(0))$.
Then $\xi$ and $u_{\xi}$ are of class $C^2$, and $(\xi,u_{\xi})$ satisfies the Herglotz' equation
\[
\begin{split}
&\frac{d}{ds} D_q L(\xi(s),s,u_{\xi}(s),\xi'(s)) \\
&=D_x L_x(\xi(s),s,u_{\xi}(s),\xi'(s))
 +D_u L(\xi(s),s,u_{\xi}(s),\xi'(s))D_q L(\xi(s),s,u_{\xi}(s),\xi'(s))
\end{split}
\]
for almost all $s \in [0,t]$.
\end{prop}

For $\xi \in \mathcal{A}_{\min}$ and $u_{\xi}$ defined in \eqref{eq:caraeq} with $u_{\xi}(0)=u_0(\xi(0))$, we define $\eta:[0,t] \to \mathbb{R}^n$ by
\[
\eta(s):=D_q L(\xi(s),s,u_{\xi}(s),\xi'(s))
\quad (s \in [0,t]).
\]
The curve $\eta$ is called the {\em dual arc associated with $\xi$}.

\begin{prop}[{\cite[Proposition 2.1(b)]{CCJWY.20}}]
Let $\xi \in \mathcal{A}_{min}$ and $u_{\xi}$ be a solution of \eqref{eq:caraeq} with $u_{\xi}(0)=u_0(\xi(0))$.
Then $\eta$ is of class $C^2$ and $(\xi,u_{\xi},\eta)$ satisfies the Lie equation
\begin{equation}
\begin{cases}
\mathrm{(a)} \ 
\xi'(s)= D_p H(\xi(s),s,u_{\xi}(s),\eta(s)), \\
\mathrm{(b)} \ 
\eta'(s)=-D_x H(\xi(s),s,u_{\xi}(s),\eta(s)) -D_uH(\xi(s),s,u_{\xi}(s),\eta(s))\eta(s), \\
\mathrm{(c)} \ 
u_{\xi}'(s)= \langle \eta(s),\xi'(s) \rangle-H(\xi(s),s,u_{\xi}(s),\eta(s))
\end{cases}
\label{eq:LHsys}
\end{equation}
for all $s \in [0,t]$.
\end{prop}

We next establish the connection between the above results and \eqref{eq:HJu}--\eqref{eq:ID}.
We define
\begin{equation}
u(x,t)
:=\inf_{\xi \in \mathcal{C}(x,t)} 
\left\{ u_0(\xi(0))+ \int_0^t L(\xi(s),s,u_{\xi}(s),\xi'(s)) \,ds \right\},
\label{eq:uvf}
\end{equation}
where $u_{\xi}$ defined \eqref{eq:caraeq} with $u_{\xi}(0)=u_0(\xi(0))$ and 
\[
\mathcal{C}(x,t)
:=\{ \xi: [0,t] \to \mathbb{R}^n \mid
\mbox{$\xi$ is absolutely continuous, $\xi(t)=x$} \}.
\]
As mentioned above, \eqref{eq:uvf} has a minimizer $\xi \in C^2([0,t])$.
Let $\mathcal{C}_{\min}(x,t)$ be the set of all the minimizers of \eqref{eq:uvf}.

\begin{prop}[{\cite[Proposition 3.1]{CCJWY.20}}]
Let $(x,t) \in \mathbb{R}^n \times (0,t)$ and $\xi \in \mathcal{C}(x,t)$.
Then, for all $\tau \in [0,t]$,
we have
\begin{equation}
u(x,t)
\leqq \int_{\tau}^t L(\xi(s),s,u_{\xi}(s),\xi'(s))\,ds +u(\xi(\tau),\tau),
\label{eq:DPP}
\end{equation}
where $u_{\xi}$ defined in \eqref{eq:caraeq} in $[\tau,t]$ with $u_{\xi}(\tau)=u(\xi(\tau),\tau)$.
Moreover, the equality holds in \eqref{eq:DPP} if and only if $\xi \in \mathcal{C}_{\min}(x,t)$.
\end{prop}

\begin{prop}[{\cite[Proposition 3.3]{CCJWY.20}, {\cite[Proposition 2.3(3)]{CH.22}}}]
The function $u(x,t)$ defined in \eqref{eq:uvf} is a viscosity solution of \eqref{eq:HJu}--\eqref{eq:ID}.
Moreover, $u$ is Lipschitz continuous and locally semiconcave in $\mathbb{R}^n \times (0,T)$.
\end{prop}

For a solution $u$, let us define 
\[
\mathcal{R}(u)
:=\{ (x,t) \in \mathbb{R}^n \times (0,T) \mid 
\mbox{$u$ is differentiable at $(x,t)$} \}.
\]

\begin{prop}[{\cite[Proposition 2.6]{CH.22}}]\label{thm:ACS20a}
Let $u$ be a viscosity solution of \eqref{eq:HJu}--\eqref{eq:ID}.
Let $(x,t) \in \mathbb{R}^n \times (0,T)$.
For $(p,\tau) \in D^* u(x,t)$, let $(\xi,\eta,u_{\xi})$ be a solution of \eqref{eq:LHsys} with the terminal condition
\begin{equation}
\xi(t)=x, \quad \eta(t)=p, \quad u_{\xi}(t)=u(x,t).
\label{eq:xeuTCp}
\end{equation}
Then $\xi \in \mathcal{C}_{\min}(x,t)$. 
Moreover, the map 
\[
D^* u(x,t) \ni (p,\tau) 
\mapsto 
\xi \in \mathcal{C}_{\min}(x,t)
\]
is bijective.
\end{prop}

\begin{prop}[{\cite[Proposition 2.3(4) and 2.4]{CCJWY.20}}]\label{prop:estLHsyssol}
Let $(x,t) \in \mathbb{R}^n \times (0,T)$.
Let $u$ be a viscosity solution of \eqref{eq:HJu}--\eqref{eq:ID} and $\xi \in \mathcal{C}_{\min}(x,t)$.
Then $(\xi(s), s) \in \mathcal{R}(u)$ for all $s \in (0,t)$.
Moreover, we have
\begin{align*}
&\eta(t) \in D_x^+ u(x,t), \\
&\eta(s)=D_x u(\xi(s),s) \quad (s \in (0,t)), \\
&\eta(0) \in D^- u_0(\xi(0)).
\end{align*}
\end{prop}

\begin{rem}
In \cite{CH.22}, the initial data $u_0$ is assumed to be smooth, but the above properties also hold even if $u_0$ is Lipschitz continuous.   
\end{rem}

\section{Approximation approach}\label{sec:ley}

In this section, we derive lower gradient estimates for viscosity solutions to \eqref{eq:HJu}--\eqref{eq:ID} using the approach based on \cite{L.01}.
More precisely, we utilize the notion of Barron--Jensen solutions, properties of the inf-convolution, and the local comparison principle.
The local comparison principle is proved in \cref{sec:LCP}.

\subsection{Properties of the inf-convolution}\label{sec:inf}

As mentioned in \cref{sec:previs}, the properties of the inf-convolution are critical for deriving lower gradient estimates.
We can extend two properties from \cite[Lemma 3.1 and 3.2]{L.01} for \eqref{eq:HJu}--\eqref{eq:ID}.
These results are used to prove \cref{thm:equivvb}.

In what follows, for $(x_0,t_0) \in \mathbb{R}^n \times (0,T)$ and $0<\rho<r$, we set
\begin{align*}
\mathcal{A}&:=B_r(x_0) \times (t_0-r,t_0+r), \\
\mathcal{A}_{\rho}&:=B_{r-\rho}(x_0) \times (t_0-r+\rho,t_0+r-\rho), \\
\mathcal{M}_{\mathcal{A}}&:=\left( 2 \max_{(y,s) \in \overline{\mathcal{A}}} |u(y,s)| \right)^{\frac{1}{2}}.
\end{align*}

\begin{lem}\label{lem:aesol}
Assume {\Hu} and {\Hc}.
Let $u \in \mathrm{Lip}_{\rm{loc}}(\mathbb{R}^n \times (0,T))$ be a viscosity subsolution (resp. supersolution) of \eqref{eq:HJu}--\eqref{eq:ID}.
Then $u$ is a viscosity supersolution (resp. subsolution) of 
\begin{equation}
-u_t(x,t)-H(x,t,u(x,t),D_x u(x,t))=0
\quad \mathrm{in} \ \mathbb{R}^n \times (0,T).
\label{eq:invereq}
\end{equation}
\end{lem}
\begin{proof}
Fix $(x_0,t_0) \in \mathbb{R}^n \times (0,T)$ and $0<\rho<r$.
Then $u \in \mathrm{Lip}(\mathcal{A})$ and we denote $L$ as a Lipschitz constant of $u$ on $\mathcal{A}$.
By Rademacher's theorem and \cite[Proposition II.1.9]{BC.97}, we have 
\begin{equation}
u_t(y,s)+H(y,s,u(y,s),D_x u(y,s)) \leqq 0
\quad \mbox{a.e. in} \ \mathcal{A}. 
\label{eq:aeeq}
\end{equation}
Let $\delta \in (0,\frac{\rho}{2})$ and $\eta_{\delta} \in C^{\infty}(\mathbb{R}^n \times \mathbb{R})$ be a standard mollifier.
We define
\[
u_{\delta}(x,t)
:=(u \ast \eta_{\delta})(x,t)
=\int_{\mathbb{R}^n \times \mathbb{R}} u(y,s)\eta_{\delta}(x-y,t-s)\,dyds
\quad ((x,t) \in \mathcal{A}_{\rho}).
\]
Multiplying $\eta_{\delta}(x-y,t-s)$ 
and integrating both sides of \eqref{eq:aeeq} on $\mathbb{R}^n \times \mathbb{R}$,
we obtain
\[
(u_{\delta})_t(x,t)+ \int_{\mathbb{R}^n \times \mathbb{R}} H(y,s,u(y,s),D_x u(y,s))\eta_{\delta}(x-y,t-s)\,dyds \leqq 0
\quad \mathrm{in} \ \mathcal{A}_{\rho}.
\]

Now we define the modulus of continuity $\mu_{\mathcal{A}}$ of $u$ in the compact set $\overline{\mathcal{A}} \times \overline{B_{\frac{\mathcal{M}_{\mathcal{A}}^2}{2}}(0)} \times \overline{B_L(0)}$ and obtain 
\[
(u_{\delta})_t(x,t)+\int_{\mathbb{R}^n \times \mathbb{R}} H(x,t,u(y,s),D_x u(y,s))\eta_{\delta}(x-y,t-s)\,dyds 
\leqq \mu_{\mathcal{A}}(2\delta).
\]
By {\Hu}, we have
\begin{equation}
\begin{split}
&(u_{\delta})_t(x,t)+ \int_{\mathbb{R}^n \times \mathbb{R}} H(x,t,u_{\delta}(x,t),D_x u(y,s))\eta_{\delta}(x-y,t-s)\,dyds \\
&\hspace{1cm} \leqq \int_{\mathbb{R}^n \times \mathbb{R}} K_3|u_{\delta}(x,t)-u(y,s)|\eta_{\delta}(x-y,t-s)\,dyds+\mu_{\mathcal{A}}(2\delta).
\end{split}
\label{eq:inteq}
\end{equation}
By the continuity of $u$,
for all $\varepsilon>0$, there exists $\delta_0>0$ such that
\[
|u(x-z,t-\tau)-u(y,s)|<\varepsilon
\]
for all $(y,s) \in B_{\delta_0}(x) \times (t-\delta_0,t+\delta_0)$ and 
$(z,\tau) \in B_{\delta_0}(0) \times (-\delta_0,\delta_0)$.
If $0<\delta<\delta_0$, then
\begin{align*}
&|u_{\delta}(x,t)-u(y,s)| \\
&=\left| \int_{B_{\delta}(x) \times (t-\delta,t+\delta)} u(z,\tau)\eta_{\delta}(x-z,t-\tau)\,dz d\tau
 -u(y,s)\int_{B_{\delta}(0) \times (-\delta,\delta)} \eta_{\delta}(z,\tau)\,dz d\tau \right| \\
&\leqq \int_{B_{\delta}(0) \times (-\delta,\delta)} |u(x-z,t-\tau)-u(y,s)|\eta_{\delta}(z,\tau)\,dz d\tau \\
&\leqq \varepsilon \int_{B_{\delta}(0) \times (-\delta,\delta)} \eta_{\delta}(z,\tau)\,dz d\tau
=\varepsilon.
\end{align*}
From \eqref{eq:inteq}, if $0<\delta<\delta_0$, then we have 
\[
(u_{\delta})_t(x,t)+ \int_{\mathbb{R}^n \times \mathbb{R}} H(x,t,u_{\delta}(x,t),D_x u(y,s))\eta_{\delta}(x-y,t-s)\,dyds
\leqq K_3\varepsilon+\mu_{\mathcal{A}}(2\delta).  
\]
By Jensen's inequality, we have
\begin{align*}
&\int_{\mathbb{R}^n \times \mathbb{R}} H(x,t,u_{\delta}(x,t),D_x u(y,s))\eta_{\delta}(x-y,t-s)\,dy ds \\
&\geqq H \left( x,t,u_{\delta}(x,t), \int_{\mathbb{R}^n \times \mathbb{R}} D_x u(y,s)\eta_{\delta}(x-y,t-s)\,dy ds \right) \\
&=H(x,t,u_{\delta}(x,t),D_x u_{\delta}(x,t))
\end{align*}
and therefore
\begin{equation}
(u_{\delta})_t(x,t)+H(x,t,u_{\delta}(x,t),D_x u_{\delta}(x,t)) 
\leqq K_3\varepsilon+\mu_{\mathcal{A}}(2\delta).
\label{eq:molieq}
\end{equation}
Since $u_{\delta} \in C^{\infty}(\mathcal{A}_{\rho})$, \cref{eq:molieq} holds in the classical sense.
A discontinuous stability result (see \cite[Chapter 2]{ABIL.11}) implies that $\displaystyle u(x,t)=\liminf_{\delta \to 0} {}_{\ast} u_{\delta}(x,t)$ is a viscosity supersolution of 
\[
-u_t(x,t)-H(x,t,u(x,t),D_x u(x,t))=-K_3\varepsilon
\quad \mathrm{in} \ \mathcal{A}_{\rho}.
\]
Letting $\varepsilon \to 0$, $u$ is a viscosity supersolution of \eqref{eq:invereq} since $\bigcup \mathcal{A}_{\rho}=\mathbb{R}^n \times (0,T)$.
\end{proof}

\begin{lem}\label{lem:infconvsol}
Assume {\Hx},  {\Hu} and {\Hc}.
Let $u \in C(\mathbb{R}^n \times (0,T))$ be a Barron--Jensen solution of \eqref{eq:HJu}--\eqref{eq:ID}.
Let $\gamma \geqq (\frac{\beta}{2}+2)C_1+K_3$, $\alpha \in (0,\frac{\rho}{2\mathcal{M}_{\mathcal{A}}})$ and $\varepsilon \in (0,\frac{e^{\frac{\gamma}{2}T}\rho}{2\mathcal{M}_{\mathcal{A}}})$.
Then the inf-convolution
\begin{equation}
u_{\varepsilon,\alpha}(x,t)
:=\inf_{(y,s) \in \overline{\mathcal{A}}}
 \left\{  u(y,s)+e^{-\gamma t}\frac{|x-y|^2}{\varepsilon^2}+\frac{|t-s|^2}{\alpha^2} \right\}  
 \label{eq:infconvt}
\end{equation}
is a viscosity subsolution of
\begin{equation}
\begin{split}
&u_t(x,t)+H(x,t,u(x,t),D_xu(x,t)) \\
&=\dfrac{C_1\beta}{2}e^{\gamma t}\varepsilon^2 +K_3\omega_{\mathcal{A}}(\mathcal{M}_{\mathcal{A}}(\varepsilon+\alpha)) +\mu_{\varepsilon,\mathcal{A}}(\mathcal{M}_{\mathcal{A}}\alpha)
\quad \mathrm{in} \ \mathcal{A}_{\rho},  
\end{split}
\label{eq:infconvapeq}
\end{equation}
where $\mu_{\varepsilon,\mathcal{A}}$ and $\omega_{\mathcal{A}}$ is, respectively, a modulus of continuity for $H$ and $u$ in the compact set $\overline{\mathcal{A}} \times \overline{B_{\frac{\mathcal{M}_{\mathcal{A}}^2}{2}}(0)} \times \overline{B_{\frac{2\mathcal{M}_{\mathcal{A}}}{\varepsilon}}(0)}$ and $\overline{\mathcal{A}}$.
\end{lem}
\begin{proof}
We denote $\mathcal{M}$ instead of $\mathcal{M}_{\mathcal{A}}$ for the sake of simplicity.
The proof is divided into two steps.

\noindent
$(1^{\circ})$
Let $\phi \in C^1(\mathbb{R}^n \times (0,T))$ such that
$u_{\varepsilon,\alpha}-\phi$ has a minimum at $(\hat{x},\hat{t}) \in \mathcal{A}_{\rho}$ and 
$(y_{\varepsilon,\alpha},s_{\varepsilon,\alpha}) \in \overline{\mathcal{A}}$ be a minimizer of $u_{\varepsilon,\alpha}(\hat{x},\hat{t})$.
Then we have
\begin{align*}
&u(y_{\varepsilon,\alpha},s_{\varepsilon,\alpha})+e^{-\gamma \hat{t}}\frac{|\hat{x}-y_{\varepsilon,\alpha}|^2}{\varepsilon^2}+\frac{|\hat{t}-s_{\varepsilon,\alpha}|^2}{\alpha^2}-\phi(\hat{x},\hat{t}) \\
&\hspace{1cm} 
 \leqq u(y,s)+e^{-\gamma t}\frac{|x-y|^2}{\varepsilon^2}+\frac{|t-s|^2}{\alpha^2}-\phi(x,t)
\end{align*}
for all $(x,t) \in  \mathcal{A}_{\rho}$ and
$(y,s) \in \overline{\mathcal{A}}$.
This implies that a function 
\[
\Phi_{\varepsilon,\alpha}(x,t,y,s):=
u(y,s)+e^{-\gamma t}\frac{|x-y|^2}{\varepsilon^2}+\frac{|t-s|^2}{\alpha^2}-\phi(x,t)
\]
attains a minimum at $(\hat{x},\hat{t},y_{\varepsilon,\alpha},s_{\varepsilon,\alpha})$ 
over $\mathcal{A}_{\rho} \times \overline{\mathcal{A}}$.
Since $(y_{\varepsilon,\alpha},s_{\varepsilon,\alpha})$ is a minimizer of $u_{\varepsilon,\alpha}(\hat{x},\hat{t})$,
we have
\begin{align*}
e^{-\gamma \hat{t}}\frac{|\hat{x}-y_{\varepsilon,\alpha}|^2}{\varepsilon^2}+\frac{|\hat{t}-s_{\varepsilon,\alpha}|^2}{\alpha^2}
&=u_{\varepsilon,\alpha}(\hat{x},\hat{t}) -u(y_{\varepsilon,\alpha},s_{\varepsilon,\alpha}) \\
&\leqq  2 \max_{(y,s) \in \overline{\mathcal{A}}} |u(y,s)|
=\mathcal{M}^2,
\end{align*}
and 
\[
|\hat{x}-y_{\varepsilon,\alpha}| 
\leqq e^{\frac{\gamma}{2}T} \mathcal{M}\varepsilon 
\leqq \frac{\rho}{2}, \quad
|\hat{t}-s_{\varepsilon,\alpha}| 
\leqq \mathcal{M}\alpha 
\leqq \frac{\rho}{2}.
\]
Therefore we obtain 
$(y_{\varepsilon,\alpha},s_{\varepsilon,\alpha}) \in \mathcal{A}$ 
and $\Phi$ has a local minimin at $(\hat{x},\hat{t},y_{\varepsilon,\alpha},s_{\varepsilon,\alpha})$.

Since $\Phi(\hat{x},\hat{t},y,s)$ has a local minimin at 
$(y_{\varepsilon,\alpha},s_{\varepsilon,\alpha})$ and
$u$ is a Barron--Jensen solution of \eqref{eq:HJu}--\eqref{eq:ID},
we have
\begin{equation}
 2\frac{\hat{t}-s_{\varepsilon,\alpha}}{\alpha^2}
 +H\left( y_{\varepsilon,\alpha},s_{\varepsilon,\alpha},u( y_{\varepsilon,\alpha},s_{\varepsilon,\alpha}),2e^{-\gamma \hat{t}}\frac{\hat{x}-y_{\varepsilon,\alpha}}{\varepsilon^2} \right)
=0.   
\label{eq:infp1}
\end{equation}
Since $\Phi(x,t,y_{\varepsilon,\alpha},s_{\varepsilon,\alpha})$ has a local minimum at $(\hat{x},\hat{t}) \in \mathcal{A}_{\rho}$,
we have
\begin{align}
&\phi_t(\hat{x},\hat{t})
=-\gamma e^{-\gamma \hat{t}}\dfrac{|\hat{x}-y_{\varepsilon,\alpha}|^2}{\varepsilon^2} 
 +2\dfrac{\hat{t}-s_{\varepsilon,\alpha}}{\alpha^2}, \label{eq:infpderi1}\\
&D_x\phi(\hat{x},\hat{t})
=2e^{-\gamma \hat{t}}\dfrac{\hat{x}-y_{\varepsilon,\alpha}}{\varepsilon^2}.  \label{eq:infpderi2}
\end{align}
Substituting \eqref{eq:infpderi1} and \eqref{eq:infpderi2} for \eqref{eq:infp1},
we obtain
\begin{equation}
\phi_t(\hat{x},\hat{t})
 +H\left(y_{\varepsilon,\alpha},s_{\varepsilon,\alpha},u( y_{\varepsilon,\alpha},s_{\varepsilon,\alpha}),D_x\phi(\hat{x},\hat{t}) \right)
=-\gamma e^{-\gamma \hat{t}}\dfrac{|\hat{x}-y_{\varepsilon,\alpha}|^2}{\varepsilon^2}.
\label{eq:infpeq}
\end{equation}
Here we have
$|u_{\varepsilon,\alpha}(\hat{x},\hat{t})| \leqq \frac{\mathcal{M}^2}{2}$ 
and 
\[
|D_x\phi(\hat{x},\hat{t})|
=2e^{-\gamma \hat{t}}\dfrac{|\hat{x}-y_{\varepsilon,\alpha}|}{\varepsilon^2}
\leqq \frac{2\mathcal{M}}{\varepsilon}.
\]
We define a modulus of continuity $\mu_{\varepsilon,\mathcal{A}}$ for $H$
in $\overline{\mathcal{A}} \times \overline{B_{\frac{\mathcal{M}^2}{2}}(0)} \times \overline{B_{\frac{2\mathcal{M}}{\varepsilon}}(0)}$
and obtain
\begin{equation}
\begin{split}
&H\left( y_{\varepsilon,\alpha},s_{\varepsilon,\alpha},u(y_{\varepsilon,\alpha},s_{\varepsilon,\alpha}),D_x\phi(\hat{x},\hat{t}) \right) \\
&\geqq 
H\left( y_{\varepsilon,\alpha},\hat{t},u(y_{\varepsilon,\alpha},s_{\varepsilon,\alpha}),D_x\phi(\hat{x},\hat{t}) \right)
 -\mu_{\varepsilon,\mathcal{A}}(\mathcal{M}\alpha).  
\end{split}
\label{eq:infpinq1}
\end{equation}
By {\Hu},
\begin{equation}
\begin{split}
&H\left( y_{\varepsilon,\alpha},\hat{t},u(y_{\varepsilon,\alpha},s_{\varepsilon,\alpha}),D_x\phi(\hat{x},\hat{t}) \right) \\
&\geqq 
H\left(y_{\varepsilon,\alpha},\hat{t},u_{\varepsilon,\alpha}(\hat{x},\hat{t}),D_x\phi(\hat{x},\hat{t}) \right)
 -K_3|u_{\varepsilon,\alpha}(\hat{x},\hat{t})--u(y_{\varepsilon,\alpha},s_{\varepsilon,\alpha})|.
\end{split}
\label{eq:infpinq2}
\end{equation}
By {\Hx} and using the Schwarz's inequality,
\begin{equation}
\begin{split}
&H\left(y_{\varepsilon,\alpha},\hat{t},u_{\varepsilon,\alpha}(\hat{x},\hat{t}),D_x\phi(\hat{x},\hat{t}) \right) \\
&\geqq 
 H\left(\hat{x},\hat{t},u_{\varepsilon,\alpha}(\hat{x},\hat{t}),D_x\phi(\hat{x},\hat{t}) \right)
 -C_1(\beta+|D_x\phi(\hat{x},\hat{t})|)|\hat{x}-y_{\varepsilon,\alpha}| \\
&\geqq 
 H\left(\hat{x},\hat{t},u_{\varepsilon,\alpha}(\hat{x},\hat{t}),D_x\phi(\hat{x},\hat{t}) \right)
 -\dfrac{C_1\beta}{2}e^{\gamma \hat{t}}\varepsilon^2 \\
&\hspace{1cm}
 -\frac{C_1\beta}{2} e^{-\gamma \hat{t}}\dfrac{|\hat{x}-y_{\varepsilon,\alpha}|^2}{\varepsilon^2}
 -2C_1e^{-\gamma \hat{t}}\frac{|\hat{x}-y_{\varepsilon,\alpha}|^2}{\varepsilon^2}.
\end{split}
\label{eq:infpinq3}
\end{equation}
Combining \eqref{eq:infpeq} to \eqref{eq:infpinq3},
we have
\begin{align*}
&\phi_t(\hat{x},\hat{t})
 +H\left(\hat{x},\hat{t}, u_{\varepsilon,\alpha}(\hat{x},\hat{t}),D_x\phi(\hat{x},\hat{t}) \right) \\
&\leqq \left( \frac{C_1\beta}{2}+2C_1-\gamma  \right) e^{-\gamma \hat{t}}\dfrac{|\hat{x}-y_{\varepsilon,\alpha}|^2}{\varepsilon^2}
 +\mu_{\varepsilon,\mathcal{A}}(\mathcal{M}\alpha)
 +\dfrac{C_1\beta}{2}e^{\gamma\hat{t}}\varepsilon^2 \\
&\hspace{1cm}
 +K_3 |u_{\varepsilon,\alpha}(\hat{x},\hat{t})-u(y_{\varepsilon,\alpha},s_{\varepsilon,\alpha})|.
\end{align*}
We denote $\omega_{\mathcal{A}}$ as a modulus of continuity for $u$ on $\overline{\mathcal{A}}$ and obtain
\begin{align*}
|u_{\varepsilon,\alpha}(\hat{x},\hat{t})-u(y_{\varepsilon,\alpha},s_{\varepsilon,\alpha})|
&=u_{\varepsilon,\alpha}(\hat{x},\hat{t})-u(y_{\varepsilon,\alpha},s_{\varepsilon,\alpha}) \\
&\leqq u(\hat{x},\hat{t})-u(y_{\varepsilon,\alpha},s_{\varepsilon,\alpha}) \\
&\leqq \omega_{\mathcal{A}}(\mathcal{M}(\varepsilon+\alpha)).
\end{align*}
Therefore if $\gamma \geqq (\frac{\beta}{2}+2)C_1+K_3$,
we have
\begin{equation}
 \begin{split}
  &\phi_t(\hat{x},\hat{t})
 +H\left(\hat{x},\hat{t}, u_{\varepsilon,\alpha}(\hat{x},\hat{t}),D_x\phi(\hat{x},\hat{t}) \right) \\
  &\leqq \frac{C_1\beta}{2}e^{\gamma\hat{t}}\varepsilon^2 
 +K_3\omega_{\mathcal{A}}(\mathcal{M}(\varepsilon+\alpha))
 +\mu_{\varepsilon,\mathcal{A}}(\mathcal{M}\alpha).   
 \end{split}
 \label{eq:infconv1}
\end{equation}

\noindent
$(2^{\circ})$
Since $u_{\varepsilon,\alpha}$ is Lipschitz continuous in $\mathcal{A}_{\rho}$,
then \cref{eq:infconv1} holds almost everywhere in $\mathcal{A}_{\rho}$ (\cite[Proposition II.1.9]{BC.97}).
By \cref{lem:aesol}, $u_{\varepsilon,\alpha}$ is a viscosity subsolution of 
\begin{align*}
&u_t(x,t)+H(x,t,u(x,t),D_x u(x,t)) \\
&=\dfrac{C_1\beta}{2}e^{\gamma t}\varepsilon^2 
 +K_3 \omega_{\mathcal{A}}(\mathcal{M}(\varepsilon+\alpha))
 +\mu_{\varepsilon,\mathcal{A}}(\mathcal{M}\alpha)
\quad \mathrm{in} \ \mathcal{A}_{\rho},   
\end{align*}
which concludes the proof.
\end{proof}

The following result immediately follows from \cref{lem:infconvsol}.
We use this result to prove \cref{thm:leylb}.
In what follows, for $r>0$, we set
\[
\mathcal{M}_r:=
\left( 2\max_{(y,s) \in \overline{B_{r}(x_0)} \times [0,T]} |u(y,s)| \right)^{\frac{1}{2}}.
\]

\begin{cor}\label{cor:infconvsp}
Assume {\Hx},  {\Hu} and {\Hc}.
Let $u \in C(\mathbb{R}^n \times [0,T))$ be a Barron--Jensen solution of \eqref{eq:HJu}--\eqref{eq:ID}.
Let $\gamma \geqq (\frac{\beta}{2}+2)C_1+K_3$, $\alpha \in (0,\frac{\rho}{2\mathcal{M}_{\mathcal{A}}})$ and
$\varepsilon \in (0,\frac{e^{\frac{\gamma}{2}T}\rho}{2\mathcal{M}_{\mathcal{A}}})$.
Then the inf-convolution
\begin{equation}
u_{\varepsilon}(x,t):=
\inf_{y \in \overline{B_{r}(x_0)}} \left\{ u(y,t)+e^{-\gamma t}\dfrac{|x-y|^2}{\varepsilon^2} \right\}
\label{eq:infconvsp}
\end{equation}
is a viscosity subsolution of
\[
u_t(x,t)+H(x,t,u(x,t),D_x u(x,t))=
\dfrac{\beta C_1}{2}e^{\gamma t}\varepsilon^2
\quad \mathrm{in} \ B_{r-\rho}(x_0) \times (0,T).
\]
\end{cor}
\begin{proof}
The proof is similar to \cref{lem:infconvsol}.
Let $\phi \in C^1(\mathbb{R}^n \times (0,T))$ such that
$u_{\varepsilon}-\phi$ has a minimum at $(\hat{x},\hat{t}) \in \mathcal{A}_{\rho}$ and
$y_{\varepsilon} \in \overline{B_r(x_0)}$ be a minimizer of $u_{\varepsilon}(\hat{x},\hat{t})$.
Then a function 
\[
\Phi_{\varepsilon}(x,y,t):=
u(y,t)+e^{-\gamma t}\frac{|x-y|^2}{\varepsilon^2}-\phi(x,t)
\]
attains a local minimum at $(\hat{x},y_{\varepsilon},\hat{t})$
and we obtain
\[
\phi_t(\hat{x},\hat{t})
 +H\left(y_{\varepsilon},\hat{t},u( y_{\varepsilon},\hat{t}),D_x\phi(\hat{x},\hat{t}) \right)
=-\gamma e^{-\gamma \hat{t}}\frac{|\hat{x}-y_{\varepsilon}|^2}{\varepsilon^2}.
\]
Since 
\[
|u_{\varepsilon}(\hat{x},\hat{t})-u(y_{\varepsilon},\hat{t})|
=e^{-\gamma \hat{t}}\frac{|\hat{x}-y_{\varepsilon}|^2}{\varepsilon^2},
\]
by {\Hx}, {\Hu} and using Schwarz's inequality,
we have 
\begin{align*}
&\phi_t(\hat{x},\hat{t})
 +H\left(\hat{x},\hat{t}, u_{\varepsilon}(\hat{x},\hat{t}),D_x\phi(\hat{x},\hat{t}) \right) \\
&\leqq \left( \frac{C_1\beta}{2}+2C_1+K_3-\gamma  \right) e^{-\gamma \hat{t}}\dfrac{|\hat{x}-y_{\varepsilon}|^2}{\varepsilon^2}
 +\frac{C_1\beta}{2}e^{\gamma\hat{t}}\varepsilon^2.
\end{align*}
Therefore if $\gamma \geqq (\frac{\beta}{2}+2)C_1+K_3$,
we have
\[
\phi_t(\hat{x},\hat{t})
 +H\left(\hat{x},\hat{t}, u_{\varepsilon,\alpha}(\hat{x},\hat{t}),D_x\phi(\hat{x},\hat{t}) \right)\leqq 
\frac{C_1\beta}{2}e^{\gamma\hat{t}}\varepsilon^2.
\]
The rest of the proof is the same as \cref{lem:infconvsol} $(2^{\circ})$.
\end{proof}

\subsection{Proof of \cref{thm:leylb}}

In this subsection, we prove \cref{thm:leylb}.
To do this, we use \cref{thm:equivvb}, \cref{cor:infconvsp}, \cref{thm:unCP}, and the following lemma.

\begin{lem}[{\cite[Lemma 4.1]{L.01}}]\label{lem:}
Assume {\Uo}.
Let $u \in C(\mathbb{R}^n \times [0,T))$ be a viscosity solution of \eqref{eq:HJu}--\eqref{eq:ID} and
$u_{\varepsilon}$ is defined in \eqref{eq:infconvsp}.
Then we have 
\begin{equation}
 u_{\varepsilon}(y,0)-u(y,0)
\leqq -\dfrac{\theta^2}{4}\varepsilon^2
\quad \mbox{for all $y \in \overline{B_r(x_0)}$.}
\label{eq:leylem}
\end{equation}
\end{lem}

\begin{proof}[Proof of \cref{thm:leylb}]
By \cref{thm:equivvb},
$u$ is also a Barron--Jensen solution of \eqref{eq:HJu}--\eqref{eq:ID}.
By \cref{cor:infconvsp},
for $0<\rho<r$, $\varepsilon \in (0,\frac{e^{\frac{\gamma}{2}T}\rho}{2M_r})$ and
$\gamma \geqq \left( \frac{\beta}{2}+2 \right)C_1+K_3$,
the inf-convolution $u_{\varepsilon}$ is a viscosity subsolution of
\[
u_t(x,t)+H(x,t,u(x,t),D_x u(x,t))
=\dfrac{C_1\beta}{2}e^{\gamma t}\varepsilon^2
\quad \mathrm{in} \ B_{r-\rho}(x_0) \times (0,T).
\]

We derive the upper and the lower estimate of $u_{\varepsilon}-u$.
On the one hand,
by \cref{thm:unCP},
we have
\[
u_{\varepsilon}(x,t)-u(x,t)
\leqq \sup_{y \in \overline{B_{r}(x_0)}} e^{-K_3 t} (u_{\varepsilon}-u)(y,0)
 +e^{-K_3 t} \int_0^t e^{K_3 s}\dfrac{\beta C_1}{2}e^{\gamma s}\varepsilon^2\,ds
\]
for all $(x,t) \in \overline{\mathcal{E}(x_0,r)} \cap (\overline{B_{r-\rho}(x_0)} \times [0,T])$,
where $\mathcal{E}(x_0,r)$ is defined in \eqref{eq:ourdod}.
By \eqref{eq:leylem} and
\begin{equation}
e^{-K_3 t} \int_0^t e^{K_3 s}\dfrac{\beta C_1}{2}e^{\gamma s}\varepsilon^2\,ds
\leqq e^{\gamma t} \int_0^t \dfrac{\beta C_1}{2}\varepsilon^2\,ds
=\dfrac{\beta C_1}{2}e^{\gamma t}t\varepsilon^2,
\label{eq:leyinte}
\end{equation}
we obtain
\begin{equation}
u_{\varepsilon}(x,t)-u(x,t)
\leqq -\frac{\theta^2}{4}e^{-K_3 t}\varepsilon^2
 +\dfrac{\beta C_1}{2}e^{\gamma t}t\varepsilon^2.
\label{eq:leylbest1}
\end{equation}

On the other hand,
let $\phi \in C^1(\mathbb{R}^n \times (0,T))$ such that
$u-\phi$ has a minimum at $(x,t) \in \mathcal{E}(x_0,r) \cap (B_{r-\rho}(x_0) \times [0,T])$.
Then there exists a continuous function $\omega: [0,\infty) \to [0,\infty)$ satisfying
$\lim_{r \to +0} \omega(r)=0$ such that
\begin{align*}
u(y,t)
&\geqq u(x,t)+\phi(y,t)-\phi(x,t) \\
&\geqq u(x,t)+\langle D_x\phi(x,t),y-x \rangle +|y-x|\omega(|y-x|)
\quad 
\mbox{for all $y \in \mathbb{R}^n$}.
\end{align*}
Let $y_{\varepsilon} \in \overline{B_r(x_0)}$ be a minimizer of $u_{\varepsilon}(x,t)$.
Then 
\begin{align*}
u_{\varepsilon}(x,t)
&=u(y_{\varepsilon},t) +e^{-\gamma t}\dfrac{|x-y_{\varepsilon}|^2}{\varepsilon^2} \\
&\geqq u(x,t)+\langle D_x\phi(x,t),y_{\varepsilon}-x \rangle
 +e^{-\gamma t}\dfrac{|x-y_{\varepsilon}|^2}{\varepsilon^2} 
  +|y_{\varepsilon}-x|\omega(|y_{\varepsilon}-x|).
\end{align*}
Let us consider 
$g(y):=\langle D_x\phi(x,t),y-x \rangle+e^{-\gamma t}\dfrac{|x-y|^2}{\varepsilon^2}$.
Then 
\begin{align*}
g(y)
&=\dfrac{e^{-\gamma t}}{\varepsilon^2} \left| \langle y-x,\frac{D_x\phi(x,t)}{2}e^{\gamma t}\varepsilon^2 \rangle \right|^2
 -\dfrac{|D_x\phi(x,t)|^2}{4}e^{\gamma t}\varepsilon^2 \\
&\geqq -\dfrac{|D_x\phi(x,t)|^2}{4}e^{\gamma t}\varepsilon^2.  
\end{align*}
Since $u \in \mathrm{Lip}(\mathbb{R}^n \times [0,T))$, we denote $L_u$ as a Lipschitz constant of $u$ in $B_r(x_0)$ and obtain
\begin{align*}
e^{-\gamma t}\dfrac{|x-y_{\varepsilon}|^2}{\varepsilon^2}
&=u_{\varepsilon}(x,t)-u(y_{\varepsilon},t) \\
&\leqq u(x,t)-u(y_{\varepsilon},t)
\leqq L_u|x-y_{\varepsilon}|,  
\end{align*}
which implies $|y_{\varepsilon}-x| \leqq L_u e^{\gamma T}\varepsilon^2$.
Therefore we have
\begin{equation}
u_{\varepsilon}(x,t)-u(x,t)
\geqq -\dfrac{|D_x\phi(x,t)|^2}{4}e^{\gamma t}\varepsilon^2 - L_ue^{\gamma T}\varepsilon^2\omega(L_ue^{\gamma T}\varepsilon^2).
\label{eq:leylbest2}
\end{equation}

Combining \eqref{eq:leylbest1} and \eqref{eq:leylbest2}, we have
\[
-\dfrac{|D_x\phi(x,t)|^2}{4}e^{\gamma t}\varepsilon^2 - L_ue^{\gamma T}\varepsilon^2\omega(L_ue^{\gamma T}\varepsilon^2)
\leqq -\frac{\theta^2}{4}e^{-K_3 t}\varepsilon^2 +\dfrac{\beta C_1}{2}e^{\gamma t}t\varepsilon^2
\]
and therefore
\[
|D_x\phi(x,t)|^2 
\geqq \theta^2 e^{-(\gamma+K_3)t}-2C_1\beta t -4L_ue^{\gamma T}\omega(L_ue^{\gamma T}\varepsilon^2)
\]
Letting $\varepsilon,\rho \to 0$, we obtain
\[
|D_x\phi(x,t)|^2 
\geqq \theta^2 e^{-(\gamma+K_3)t}-2C_1\beta t
\]
for all $(x,t) \in \mathcal{E}(x_0,r)$.
Taking $\gamma=\left( \frac{\beta}{2}+2 \right)C_1+K_3$, we complete the proof.
\end{proof}

\begin{rem}\label{rem:leyobrem}
We can derive a sharper estimate by directly calculating the integral term in \eqref{eq:leyinte}.
In fact, we have 
\begin{align*}
e^{-K_3 t} \int_0^t e^{K_3 s}\dfrac{\beta C_1}{2}e^{\gamma s}\varepsilon^2\,ds
&=\dfrac{\beta C_1}{2}\varepsilon^2 e^{-K_3t} \int_0^t e^{(\gamma+K_3)s}\,ds \\
&=\dfrac{\beta C_1}{2(\gamma+K_3)}\varepsilon^2 e^{-K_3t}(e^{(\gamma+K_3)t}-1) \\
&=\dfrac{\beta C_1}{2(\gamma+K_3)}\varepsilon^2(e^{\gamma t}-e^{-K_3t})
\label{eq:leyinte} 
\end{align*}
and therefore, we obtain
\begin{equation}
u_{\varepsilon}(x,t)-u(x,t)
\leqq -\frac{\theta^2}{4}e^{-K_3 t}\varepsilon^2
 +\dfrac{\beta C_1}{2(\gamma+K_3)}\varepsilon^2(e^{\gamma t}-e^{-K_3t})
\label{eq:leylbest1betu}
\end{equation}
instead of \eqref{eq:leylbest1}.
Combining \eqref{eq:leylbest1betu} and \eqref{eq:leylbest2} and then letting $\varepsilon,\rho \to 0$, we have
\[
|D_x\phi(x,t)|^2 
\geqq \theta^2 e^{-(\gamma+K_3)t}-\dfrac{2\beta C_1}{\gamma+K_3} (1-e^{-(\gamma+K_3) t}).
\]
Taking $\gamma=\left( \frac{\beta}{2}+2 \right)C_1+K_3$, we obtain
\[
|D_x\phi(x,t)|^2 
\geqq \theta^2 e^{-((\frac{\beta}{2}+2)C_1+2K_3)t}-\dfrac{4\beta C_1}{(\beta+4)C_1+4K_3} (1-e^{-((\frac{\beta}{2}+2)C_1+2K_3)t}).
\]
\end{rem}

\section{Dynamical approach}\label{sec:our}

In this section, we derive gradient estimates for solutions to \eqref{eq:HJu}--\eqref{eq:ID} using the approach based on \cite{HH.23}.
More precisely, we use the Lie equation and study how the initial gradients propagate along the solution of \eqref{eq:LHsys}.

\subsection{Estimates of solutions to the Lie equation}

In this subsection, We derive some estimates for a solution $(\xi,\eta,u_{\xi})$ of \eqref{eq:LHsys}--\eqref{eq:xeuTCp}
if $u$ is differentiable at $(x,t)$.
That is, we solve \eqref{eq:LHsys} with the terminal condition
\begin{equation}
\xi(t)=x, \quad \eta(t)=D_x u(x,t), \quad u_{\xi}(t)=u(x,t).
\label{eq:xeuTC}
\end{equation}

\begin{prop}\label{prop:estRu1}
Assume {\Hx}--{\Hr} and {\Hsc}.
Let $u \in C(\mathbb{R}^n \times [0,T))$ be a viscosity solution of \eqref{eq:HJu}--\eqref{eq:ID}.
Let $(x,t) \in \mathcal{R}(u)$
and $(\xi,\eta,u_{\xi})$ be a solution of \eqref{eq:LHsys}--\eqref{eq:xeuTC}.
Then if $(C_1,K_3) \neq (0,0)$,
\begin{equation}
\begin{cases}
|D_x u(x,t)-\eta(0)|
\leqq \left( \frac{C_1\beta}{C_1+K_3}+|D_x u(x,t)| \right)(e^{(C_1+K_3)t}-1), \\
|D_x u(x,t)-\eta(0)|
\leqq \left( \frac{C_1\beta}{C_1+K_3}+|\eta(0)| \right)(e^{(C_1+K_3)t}-1),
\end{cases}
\label{eq:estDu1}
\end{equation}
where $C_1$ and $\beta$ are the constants in {\Hx} and
\begin{equation}
\begin{split}
&\phantom{\leqq} |\eta(0)|e^{-(C_1+K_3)t}-\frac{C_1\beta}{C_1+K_3}(1-e^{-(C_1+K_3)t}) \\
&\leqq |D_x u(x,t)| \\
&\leqq |\eta(0)|e^{(C_1+K_3)t}+\frac{C_1\beta}{C_1+K_3}(e^{(C_1+K_3)t}-1).
\end{split}
\label{eq:estDu2}
\end{equation}
Moreover, 
\begin{gather}
\begin{split}
&\phantom{\leqq} |\eta(0)| e^{-C_1t} -\beta(1-e^{-C_1t}) \\
&\leqq |D_x u(x,t)| \\
&\leqq |\eta(0)| e^{C_1t}+\beta(e^{C_1t}-1)
\end{split}
\quad \mbox{if $K_3=0$}, \label{eq:estDuk} \\
|\eta(0)| e^{-(C_1+K_3)t} 
\leqq |D_x u(x,t)|
\leqq |\eta(0)| e^{(C_1+K_3)t}
\quad \mbox{if $\beta=0$}, \label{eq:estDub} \\
|\eta(0)| e^{-K_3t} 
\leqq |D_x u(x,t)|
\leqq |\eta(0)| e^{K_3t}
\quad \mbox{if $C_1=0$}, \label{eq:estDuc} \\
D_x u(x,t)=\eta(0)
\quad \mbox{if $(C_1,K_3)=(0,0)$}. \label{eq:estDuck}
\end{gather}
\end{prop}

\begin{rem}
We note that \cref{eq:estDuk} is also obtained in \cite[Proposition 3.1]{HH.23} when $H$ does not depend on the variable $u$.
\end{rem}

\begin{proof}
It suffices  to prove \eqref{eq:estDu1}
since \eqref{eq:estDu2} to \eqref{eq:estDuck} are immediate from \eqref{eq:estDu1}.
Let $\tau \in [0,t)$.
Integrating both the sides of \eqref{eq:LHsys}(b) over $[\tau,t]$,
we have
\begin{align*}
\eta(t)-\eta(\tau)
&=-\int_{\tau}^t D_x H(\xi(s),s,u_{\xi}(s),\eta(s))\,ds \\
&\hspace{2cm} -\int_{\tau}^t D_u H(\xi(s),s,u_{\xi}(s),\eta(s))\eta(s)\,ds.  
\end{align*}
By {\Hx} and {\Hu}, 
\begin{align*}
|\eta(t)-\eta(\tau)|
&\leqq \int_{\tau}^t |D_x H(\xi(s),s,u_{\xi}(s),\eta(s))|\,ds  \\
&\hspace{2cm} +\int_{\tau}^t |D_u H(\xi(s),s,u_{\xi}(s),\eta(s))||\eta(s)| \,ds \\
&\leqq \int_{\tau}^t C_1(\beta +|\eta(s)|) \,ds +\int_{\tau}^t K_3|\eta(s)|\,ds \\
&\leqq \int_{\tau}^t (C_1\beta+(C_1+K_3)(|\eta(t)|+|\eta(t)-\eta(s)|))\,ds \\
&=(C_1\beta+(C_1+K_3)|\eta(t)|)(t-\tau) \\
&\hspace{2cm} +(C_1+K_3)\int_{\tau}^t |\eta(t)-\eta(s)|\,ds.
\end{align*}
Applying Gronwall's lemma and then taking $\tau=0$, we have
\begin{align*}
&|\eta(t)-\eta(0)| \\
&\leqq (C_1\beta+(C_1+K_3)|\eta(t)|)t \\
&\phantom{\leqq}  +e^{(C_1+K_3)t} \int_0^t e^{-(C_1+K_3)(t-s)} (C_1+K_3) (C_1\beta+(C_1+K_3)|\eta(t)|)(t-s) \, ds \\
&=(C_1\beta+(C_1+K_3)|\eta(t)|) \left\{ t +(C_1+K_3) \int_0^t (t-s)e^{(C_1+K_3)s} \, ds \right\}.
\end{align*}
If $(C_1,K_3) \neq (0,0)$, 
\begin{align*}
&(C_1+K_3) \int_0^t (t-s) e^{(C_1+K_3)s} \, ds \\
&=\left[ (t-s) e^{(C_1+K_3)s} \right]_0^t+\int_0^t e^{(C_1+K_3)s} \, ds \\
&=-t+\left[ \frac{e^{(C_1+K_3)s}}{C_1+K_3} \right]_0^t \\
&=-t+ \frac{e^{(C_1+K_3)t}-1}{C_1+K_3},
\end{align*}
and therefore, we obtain
\[
|\eta(t)-\eta(0)|
\leqq \left( \frac{C_1\beta}{C_1+K_3}+|\eta(t)| \right) (e^{(C_1+K_3)t}-1),
\]
which is the first estimate in \eqref{eq:estDu1}
since $\eta(t)=D_x u(x,t)$.

Let $\sigma \in (0,t]$.
Integrating both the sides of \eqref{eq:LHsys}(b) over $[0,\sigma]$ and then taking $\sigma=t$,
we derive the second estimate in \eqref{eq:estDu1} in a similar way.
\end{proof}

We also use some results obtained in \cite{HH.23}.
These results are true in our settings.

\begin{prop}[{\cite[Proposition 3.2]{HH.23}}]\label{prop:estatRu2}
Assume {\Hx}--{\Hr} and {\Hsc}.
Let $u \in C(\mathbb{R}^n \times [0,T))$ be a viscosity solution of \eqref{eq:HJu}--\eqref{eq:ID}.
Let $(x,t) \in \mathcal{R}(u)$
and $(\xi,\eta,u_{\xi})$ be a solution of \eqref{eq:LHsys}--\eqref{eq:xeuTC}.
Then 
\[
\begin{cases}
|x-\xi(0)|
\leqq \left( \dfrac{B_2}{A_2}+|x| \right)(e^{A_2t}-1)
& \mbox{if $A_2>0$,} \\
|x-\xi(0)|
\leqq B_2 t
& \mbox{if $A_2=0$,}
\end{cases}
\]
where $A_2$ and $B_2$ are the constants in {\Hp}.
\end{prop}

\begin{lem}[{\cite[Lemma 3.5]{HH.23}}]\label{lem:estball}
Assume {\Hx}--{\Hr} and {\Hsc}.
Let $(x,t) \in \mathbb{R}^n \times (0,T)$
and $\xi \in \mathcal{C}_{\min}(x,t)$.
Then, 
\[
B_{R(\xi(s),s)}(\xi(s)) \subset B_{R(x,t)}(x)
\quad (s \in (0,t)),
\]
where $R(x,t)$ is defined in \eqref{eq:defiR}.
\end{lem}

By \cref{prop:estLHsyssol}, \cref{prop:estRu1}, \cref{prop:estatRu2} and \cref{lem:estball},
we obtain the following theorem under the additional assumptions {\Hr} and {\Hsc}.

\begin{thm}\label{thm:sHumain}
Assume {\Hx}--{\Hr} and {\Hsc}.
Let $u \in C(\mathbb{R}^n \times [0,T))$ be a viscosity solution of \eqref{eq:HJu}--\eqref{eq:ID}.
Then
\begin{enumerate}
\item 
if $(C_1,K_3) \neq (0,0)$, 
\[
\begin{split}
&\phantom{\leqq} I(x,t;u_0)e^{-(C_1+K_3)t}-\frac{C_1\beta}{C_1+K_3}(1-e^{-(C_1+K_3)t}) \\
&\leqq |p| \\
&\leqq S(x,t;u_0)e^{(C_1+K_3)t}+\frac{C_1\beta}{C_1+K_3}(e^{(C_1+K_3)t}-1)
\end{split}
\]
for all $(x,t) \in \mathbb{R}^n \times (0,T)$ and $p \in D^-_x u(x,t)$,
where 
\begin{align*}
S(x,t;u_0)
&:=\sup \left\{ |p| \relmiddle| 
p \in D^- u_0(y), \, y \in \overline{B_{R(x,t)}(x)} \right\}, \\
I(x,t;u_0)
&:=\inf \left\{ |p| \relmiddle| 
p \in D^- u_0(y), \, y \in \overline{B_{R(x,t)}(x)} \right\}.
\end{align*}

\item 
if $(C_1,K_3)=(0,0)$, 
\[
I(x,t;u_0)
\leqq |p|
\leqq S(x,t;u_0)
\]
for all $(x,t) \in \mathbb{R}^n \times (0,T)$ and $p \in D^-_x u(x,t)$.
\end{enumerate}
\end{thm}
The proof is similar to \cite[Proof of Theorem 3.6(1)]{HH.23}, 
so we omit it.

\subsection{Proof of \cref{thm:ourlb}}

In this subsection, we prove \cref{thm:ourlb}.
To do this, we approximate $H$ by $H_{\varepsilon}$ satisfying {\Hr} and {\Hsc}.
For more details, see \cite[Section 2.3]{HH.23}.

First, we extend $H$ as
\[
H(x,t,u,p)
=\begin{cases}
H(x,0,u,p) & (t<0), \\
H(x,T,u,p) & (t>T).
\end{cases}
\]
Let $\varepsilon \in (0,1]$.
We define $H_{\varepsilon}: \mathbb{R}^n \times \mathbb{R} \times \mathbb{R} \times \mathbb{R}^n \to \mathbb{R}$ by
\begin{equation}
H_{\varepsilon}(x,t,u,p)
=(H*\rho_{\varepsilon})(x,t,u,p)+\varepsilon \sqrt{|p|^2+1},
\label{eq:defHe}
\end{equation}
where $\rho_{\varepsilon}$ is the standard Friedrichs mollifier and 
\[
(H*\rho_{\varepsilon})(x,t,u,p)
=(H*\rho_{\varepsilon})(z)
:=\int_{B_{\varepsilon}(0)} H(z-w)\rho_{\varepsilon}(w) \, dw.
\]
Then $H_{\varepsilon}$ satisfies {\Hr} and {\Hsc}.
Furthermore, {\Hx} and {\Hu} respectively hold for $H_{\varepsilon}$ with the same constant $C_1$ and $K_3$.
{\Hp} also holds for $H_{\varepsilon}$, 
but the Lipschitz constant must be replaced by $A_2|x|+B_2+\varepsilon$.
For this reason, we prepare the following notations
\[
R_{\varepsilon}(x,t):=\begin{cases} 
\left( \dfrac{B_2+\varepsilon}{A_2}+|x| \right)(e^{A_2t}-1) & \mbox{if $A_2>0$,} 
\smallskip \\ 
(B_2+\varepsilon) t & \mbox{if $A_2=0$} 
\end{cases} 
\]
and 
\begin{align*}
S_{\varepsilon}(x,t;u_0)
&:=\sup \left\{ |p| \relmiddle| 
p \in D^- u_0(y), \, y \in \overline{B_{R_{\varepsilon}(x,t)}(x)} \right\}, \\
I_{\varepsilon}(x,t;u_0)
&:=\inf \left\{ |p| \relmiddle| 
p \in D^- u_0(y), \, y \in \overline{B_{R_{\varepsilon}(x,t)}(x)} \right\}.
\end{align*}

\begin{proof}[Proof of \cref{thm:ourlb}]
Let $\varepsilon \in (0,1]$ and $H_{\varepsilon}$ defined in \eqref{eq:defHe}.
Let $u_{\varepsilon}$ be a viscosity solution of \eqref{eq:HJuap}--\eqref{eq:ID}.
By the standard stability results for viscosity solutions (see \cite[Section 6]{CIL.92}),
$u_{\varepsilon}$ converges to $u$ locally uniformly in $\mathbb{R}^n \times [0,T)$ as $\varepsilon \to +0$.

Let $p \in D_x^- u(x,t)$.
By \cite[Lemma II.2.4]{BC.97}, 
there exist 
$\{ (x_{\varepsilon},t_{\varepsilon}) \}_{\varepsilon \in (0,1]} \subset \mathbb{R}^n \times (0,T)$ and 
$\{ p_{\varepsilon} \}_{\varepsilon \in (0,1]} \subset \mathbb{R}^n$
such that 
\[
p_{\varepsilon} \in D_x^- u_{\varepsilon}(x_{\varepsilon},t_{\varepsilon}), \quad 
\lim_{\varepsilon \to 0} (x_{\varepsilon},t_{\varepsilon},p_{\varepsilon})=(x,t,p).
\]
We now apply \cref{thm:sHumain} to obtain 
\[
\begin{split}
&\phantom{\leqq} I_{\varepsilon}(x_{\varepsilon},t_{\varepsilon};u_0)e^{-(C_1+K_3)t_{\varepsilon}}-\frac{C_1\beta}{C_1+K_3}(1-e^{-(C_1+K_3)t_{\varepsilon}}) \\
&\leqq |p| \\
&\leqq S_{\varepsilon}(x_{\varepsilon},t_{\varepsilon};u_0)e^{(C_1+K_3)t_{\varepsilon}}+\frac{C_1\beta}{C_1+K_3}(e^{(C_1+K_3)t_{\varepsilon}}-1).
\end{split}
\]
For an arbitrary $\delta>0$,
we have $R_{\varepsilon}(x_{\varepsilon},t_{\varepsilon}) < R(x,t)+\delta$ if $\varepsilon$ small enough and obtain
\begin{align*}
&\phantom{\leqq} \inf \left\{ |q| \relmiddle| 
q \in D^- u_0(y),~y \in \overline{B_{R(x,t)+\delta}(x)} \right\} \cdot e^{-(C_1+K_3)t_{\varepsilon}} \\
&\hspace{5cm} -\frac{C_1\beta}{C_1+K_3}(1-e^{-(C_1+K_3)t_{\varepsilon}}) \\
&\leqq |p_{\varepsilon}| \\
&\leqq \sup \left\{ |q| \relmiddle| 
q \in D^- u_0(y),~y \in \overline{B_{R(x,t)+\delta}(x)} \right\} \cdot e^{(C_1+K_3)t_{\varepsilon}} \\
&\hspace{5cm} +\frac{C_1\beta}{C_1+K_3}(1-e^{(C_1+K_3)t_{\varepsilon}}).
\end{align*}
We obtain the desired inequalities by letting $\varepsilon \to +0$ and $\delta \to +0$.
\end{proof}

\section{Comparison with two approach}\label{sec:cp}

In this section,
we compare \cref{thm:leylb} with \cref{thm:ourleylb} under the assumption {\Uo}.
First, the domain of dependence $\mathcal{E}(x_0,r)$, where gradient estimates are available, is the same.
Although the lower bound is the same if $(C_1, K_3)=(0,0)$,
the lower bound is generally not the same.
We find out which the lower bound is larger.

In what follows, we only consider the case of $(C_1, K_3) \neq (0,0)$.
We introduce some notations.

\begin{defi}\label{defi:lb}
Let $\theta>0$ and assume $(C_1, K_3) \neq (0,0)$.
We define 
\begin{align*}
l(t)&:=\sqrt{\theta^2e^{-((\frac{\beta}{2}+2)C_1+2K_3)t}-2C_1\beta t}, \\
L(t)&:=\theta e^{-(C_1+K_3)t}-\frac{C_1\beta}{C_1+K_3}(1-e^{-(C_1+K_3)t}),
\end{align*}
where $\beta$ and $C_1$ are the constants in {\Hx} and $K_3$ is a constant in {\Hu}.
Moreover,
if $\beta=1$, then 
we define $t_l,t_L>0$ as 
\[
t_l:=\inf \{ t \in [0,\infty) \mid l(t)=0 \}, \quad
t_L:=\inf \{ t \in [0,\infty) \mid L(t)=0 \}.
\]
\end{defi}

The next theorem shows that 
\cref{thm:ourleylb} is better result than \cref{thm:leylb}.

\begin{thm}\label{thm:complb}
Let $\theta>0$ and assume $(C_1, K_3) \neq (0,0)$.
\begin{enumerate}
\item 
If $\beta=0$, then $l(t)=L(t)$ in $[0,\infty)$.

\item 
If $\beta=1$, then $l(0)=L(0)$ and $l(t)<L(t)$ in $(0,t_l]$
\end{enumerate}
\end{thm}
\begin{proof}
(1) 
It is obvious since
\[
l(t)=L(t)=\theta e^{-(C_1+K_3)t}.
\]

\noindent
(2)
When $\beta=1$, we have
\begin{align*}
l(t)&=\sqrt{\theta^2e^{-(\frac{5}{2}C_1+2K_3)t}-2C_1t}, \\
L(t)&=\theta e^{-(C_1+K_3)t}-\frac{C_1}{C_1+K_3}(1-e^{-(C_1+K_3)t}).
\end{align*}
It suffices to prove $\{ l(t) \}^2<\{ L(t) \}^2$ in $(0,t_l]$ since $l(t) \geqq 0$ in $(0,t_l]$.
We set
\begin{align*}
F(t)
&:=\{ L(t) \}^2-\{ l(t) \}^2.
\end{align*}
We calculate
\begin{align*}
\{ l(t) \}^2&=\theta^2e^{-(\frac{5}{2}C_1+2K_3)t}-2C_1t, \\
\{ L(t) \}^2&=\theta^2 e^{-2(C_1+K_3)t}-\frac{2C_1\theta}{C_1+K_3}e^{-(C_1+K_3)t}(1-e^{-(C_1+K_3)t}) \\
&\hspace{4cm}
 +\frac{C_1^2}{(C_1+K_3)^2}(1-e^{-(C_1+K_3)t})^2,
\end{align*}
and therefore
\begin{align*}
F(t)
&=\theta^2 e^{-2(C_1+K_3)t}(1-e^{-\frac{1}{2}C_1t}) \\
&\hspace{1cm} -\frac{2C_1\theta}{C_1+K_3}e^{-(C_1+K_3)t}(1-e^{-(C_1+K_3)t}) \\
&\hspace{2cm} +\frac{C_1^2}{(C_1+K_3)^2}(1-e^{-(C_1+K_3)t})^2+2C_1t \\
&=e^{-2(C_1+K_3)t}(1-e^{-\frac{1}{2}C_1t}) \left\{ 
\theta-\frac{C_1}{C_1+K_3}e^{(C_1+K_3)t}\frac{1-e^{-(C_1+K_3)t}}{1-e^{-\frac{1}{2}C_1t}} \right\}^2 \\
&\hspace{1cm} -\frac{C_1^2}{(C_1+K_3)^2}\frac{(1-e^{-(C_1+K_3)t})^2}{1-e^{-\frac{1}{2}C_1t}} \\
&\hspace{2cm} +\frac{C_1^2}{(C_1+K_3)^2}(1-e^{-(C_1+K_3)t})^2+2C_1t \\
&\geqq  
-\frac{C_1^2}{(C_1+K_3)^2} \cdot \frac{e^{-\frac{1}{2}C_1t}(1-e^{-(C_1+K_3)t})^2 }{1-e^{-\frac{1}{2}C_1t}}+2C_1t
=:G(t).
\end{align*}
Since $e^X>X+1$ ($X \neq 0$), we have
\begin{align*}
&(1-e^{-\frac{1}{2}C_1t})G(t) \\
&=-\frac{C_1^2}{(C_1+K_3)^2} \cdot e^{-\frac{1}{2}C_1t}(1-e^{-(C_1+K_3)t})^2+2C_1t(1-e^{-\frac{1}{2}C_1t}) \\
&>-\frac{C_1^2}{(C_1+K_3)^2} \cdot e^{-\frac{1}{2}C_1t}(C_1+K_3)^2t^2+2C_1t(1-e^{-\frac{1}{2}C_1t}) \\
&=2C_1t \left\{ -\left( \frac{1}{2}C_1t+1 \right)e^{-\frac{1}{2}C_1t} +1 \right\} \\
&>2C_1t (-e^{\frac{1}{2}C_1t} \cdot e^{-\frac{1}{2}C_1t} +1) \\
&=0.
\end{align*}
Since $1-e^{-\frac{1}{2}C_1t}>0$ in $(0,t_l]$,
we conclude that
\[
F(t) \geqq G(t) >0,
\]
which completes the proof.
\end{proof}

\section{Examples}\label{sec:ex}

In this section, we show some examples and explain our result is optimal in some senses. 

\subsection{Special case}

We first consider the following equation
\begin{equation}
u_t(x,t)+\lambda u(x,t)+H_0(x,t,D_x u(x,t))=0 
\quad \mathrm{in} \ \mathbb{R}^n \times (0,T).
\label{eq:geneHJ}
\end{equation}
Here is a Hamiltonian $H(x,t,u,p)=\lambda u+H_0(x,t,p)$ for $\lambda \in \mathbb{R}$ and $H_0: \mathbb{R}^n \times [0,T] \times \mathbb{R}^n \to \mathbb{R}$ satisfying {\Hx}, {\Hp}, and {\Hc}.
$H$ satisfies {\Hu} with $K_3=|\lambda|$, but \cref{thm:ourlb} may not be an optimal estimate if $\lambda<0$ (see \cref{ex:trans}).
In this case, we can obtain another gradient estimate for \eqref{eq:geneHJ} by a direct calculation.

Let $u$ be a viscosity solution of \eqref{eq:geneHJ}--\eqref{eq:ID}.
By \cite[Exercise 2.3]{BC.97},
$v(x,t):=e^{\lambda t} u(x,t)$ is a viscosity solution of
\[
v_t(x,t)+F(x,t,D_x v(x,t))=0 \quad \mathrm{in} \ \mathbb{R}^n \times (0,T),
\]
where $F(x,t,p):=e^{\lambda t} H_0(x,t,e^{-\lambda t}p)$.
If $H_0$ is positively homogeneous with respect to $p$, i.e.
\begin{equation}
H_0(x,t,\mu p) =\mu H_0(x,t,p)
\quad \mbox{for all $(x,t,p) \in \mathbb{R}^n \times [0,T] \times \mathbb{R}^n$
and $\mu \geqq 0$},
\label{eq:1-homo}
\end{equation}
then we have $F=H_0$ on $\mathbb{R}^n \times [0,T] \times \mathbb{R}^n$.
Applying \cref{thm:ourprelb}, we obtain 
\[
|p| \geqq 
\underline{I}(x,t;u_0) e^{-C_1t} -\beta (1-e^{-C_1t})
\]
for all $(x,t) \in \mathcal{E}(x_0,r)$ and $p \in D_x^- v(x,t)$.
Since $D_x^- v(x,t)=e^{\lambda t}D_x^- u(x,t)$, we have
\begin{equation}
|p| \geqq 
\underline{I}(x,t;u_0) e^{-(C_1+\lambda)t} -\beta (e^{-\lambda t}-e^{-(C_1+\lambda)t})
\label{eq:levellb}
\end{equation}
for all $(x,t) \in \mathcal{E}(x_0,r)$ and $p \in D_x^- u(x,t)$.

More generally, we have the following.
\begin{prop}\label{prop:estatRu1lam}
Assume that $H_0$ satisfies {\Hx}, {\Hp}, {\Hc}, {\Hr}, and {\Hsc}.
Let $u$ be a viscosity solution of \eqref{eq:geneHJ}--\eqref{eq:ID}.
Let $(x,t) \in \mathcal{R}(u)$ and $(\xi,\eta) \in C^1([0,t])^2$ be a solution of \eqref{eq:LHsys}--\eqref{eq:xeuTC}.
Then
\begin{equation}
\begin{cases}
\begin{aligned}
|D_x u(x,t)-e^{-\lambda t}\eta(0)|
\leqq \frac{C_1\beta}{C_1+\lambda}(e^{C_1 t}-e^{-\lambda t})+|D_x u(x,t)|(e^{C_1t}-1) \\
\mbox{if $\lambda \neq -C_1$},
\end{aligned} \\ 
|D_x u(x,t)-e^{C_1 t}\eta(0)|
\leqq C_1\beta te^{C_1 t} +|D_x u(x,t)|(e^{C_1t}-1) 
\quad \mbox{if $\lambda=-C_1$},
\end{cases} 
\label{eq:etalam1}
\end{equation}
and 
\begin{equation}
\begin{cases}
\begin{aligned}
|D_x u(x,t)-e^{-\lambda t}\eta(0)|
\leqq \frac{C_1\beta}{C_1-\lambda}(e^{(C_1-\lambda) t}-1)+|\eta(0)|(e^{(C_1-\lambda) t}-e^{-\lambda t}) \\
\mbox{if $\lambda \neq C_1$}, 
\end{aligned} \\
|D_x u(x,t)-e^{-C_1 t}\eta(0)|
\leqq C_1\beta t +|\eta(0)|(1-e^{-C_1 t}) 
\quad \mbox{if $\lambda=C_1$}, 
\end{cases}
\label{eq:etalam2}
\end{equation}
where $C_1$ and $\beta$ are the constants in {\Hx}.
Moreover,
\begin{align*}
|D_x u(x,t)| &\geqq |\eta(0)|e^{-(C_1+\lambda)t}-\frac{C_1\beta}{C_1+\lambda}(1-e^{-(C_1+\lambda)t}) \quad \mbox{if $\lambda \neq -C_1$}, \\
|D_x u(x,t)| &\geqq |\eta(0)|-C_1\beta t \quad \mbox{if $\lambda=-C_1$}, \\
|D_x u(x,t)| &\leqq |\eta(0)|e^{(C_1-\lambda)t}+\frac{C_1\beta}{C_1-\lambda}(e^{(C_1-\lambda)t}-1) \quad \mbox{if $\lambda \neq C_1$}, \\
|D_x u(x,t)| &\leqq |\eta(0)|+C_1\beta t \quad \mbox{if $\lambda=C_1$}.
\end{align*}
\end{prop}
\begin{proof}
We can calculate in a similar way to \cref{prop:estRu1}, so we only prove the first inequality of \eqref{eq:etalam1}.
By \eqref{eq:LHsys}(b), we have
\begin{align*}
\eta'(s)
&=-D_x H (\xi(s),s,u_{\xi}(s),\eta(s))-D_u H (\xi(s),s,u_{\xi}(s),\eta(s))\eta(s) \\
&=-D_x H_0 (\xi(s),s,\eta(s))-\lambda \eta(s).
\end{align*}
Then $\zeta(s):=e^{\lambda s}\eta(s)$ is a solution of 
\begin{equation}
\zeta'(s)
=-e^{\lambda s} D_x H_0 (\xi(s),s,e^{-\lambda s}\zeta(s)).
\label{eq:estatRulam}
\end{equation}

Let $\tau \in [0,t)$.
Integrating both the sides of \eqref{eq:estatRulam} over $[\tau,t]$ and by {\Hx}, we have
\begin{align*}
|\zeta(t)-\zeta(\tau)|
&\leqq \int_{\tau}^t  e^{\lambda s}|D_x H_0 (\xi(s),s,e^{-\lambda s}\zeta(s))|\,ds \\
&\leqq \int_{\tau}^t  C_1e^{\lambda s} (\beta+|e^{-\lambda s}\zeta(s))|)\,ds \\
&=C_1 \int_{\tau}^t (\beta e^{\lambda s}+|\zeta(s)|)\,ds \\
&\leqq C_1 \int_{\tau}^t (\beta e^{\lambda s}+|\zeta(t)|+|\zeta(t)-\zeta(s)|) \,ds \\
&=\frac{C_1\beta}{\lambda}(e^{\lambda t}-e^{\lambda \tau})+C_1|\zeta(t)|(t-\tau)
 +C_1\int_{\tau}^t |\zeta(t)-\zeta(s)| \,ds. 
\end{align*}
Applying Gronwall's lemma and letting $\tau=0$, we have
\begin{align*}
&|\zeta(t)-\zeta(0)| \\
&\leqq \frac{C_1\beta}{\lambda}(e^{\lambda t}-1)+C_1|\zeta(t)|t \\
&\hspace{0.5cm} +e^{C_1t} \int_0^t e^{-C_1(t-s)} C_1 \left( \frac{C_1\beta}{\lambda}(e^{\lambda t}-e^{\lambda s})+C_1|\zeta(t)|(t-s) \right) \,ds \\
&=\frac{C_1\beta}{\lambda}(e^{\lambda t}-1)+C_1|\zeta(t)|t \\
&\hspace{0.5cm} 
 +\frac{C_1^2\beta}{\lambda} \int_0^t (e^{\lambda t}e^{C_1s}-e^{(C_1+\lambda)s})\,ds
 +C_1^2|\zeta(t)| \int_0^t (t-s)e^{C_1s}\,ds.
\end{align*}
If $\lambda \neq -C_1$, we compute 
\begin{align*}
&\frac{C_1^2\beta}{\lambda} \int_0^t (e^{\lambda t}e^{C_1s}-e^{(C_1+\lambda)s})\,ds \\
&=\frac{C_1\beta}{\lambda} \left(  
 e^{\lambda t} (e^{C_1 t}-1)-\frac{C_1}{C_1+\lambda}(e^{(C_1+\lambda)t}-1) \right) \\
&=\frac{C_1\beta}{\lambda}(e^{(C_1+\lambda)t}-e^{\lambda t})
 -\frac{C_1^2\beta}{\lambda(C_1+\lambda)}(e^{(C_1+\lambda)t}-1)
\end{align*}
and
\begin{align*}
&C_1^2|\zeta(t)| \int_0^t (t-s)e^{C_1s}\,ds \\
&=C_1|\zeta(t)| \left[ (t-s) e^{C_1s} \right]_0^t
 +C_1|\zeta(t)| \int_0^t e^{C_1s}\,ds \\
&=-C_1|\zeta(t)|t+|\zeta(t)|(e^{C_1t}-1).
\end{align*}
Combining them, we obtain
\begin{align*}
&|\zeta(t)-\zeta(0)| \\
&\leqq \frac{C_1\beta}{\lambda}(e^{\lambda t}-1)+C_1|\zeta(t)|t
+\frac{C_1\beta}{\lambda}(e^{(C_1+\lambda)t}-e^{\lambda t}) \\
&\hspace{0.5cm}
 -\frac{C_1^2\beta}{\lambda(C_1+\lambda)}(e^{(C_1+\lambda)t}-1)
  -C_1|\zeta(t)|t +|\zeta(t)|(e^{C_1t}-1) \\
&=\frac{C_1\beta}{C_1+\lambda}(e^{(C_1+\lambda)t}-1)+|\zeta(t)|(e^{C_1t}-1).
\end{align*}
Since $\zeta(s)=e^{\lambda s}\eta(s)$, we have
\begin{align*}
|e^{\lambda t}\eta(t)-\eta(0)|
\leqq \frac{C_1\beta}{C_1+\lambda}(e^{(C_1+\lambda)t}-1)+|e^{\lambda t}\eta(t)|(e^{C_1t}-1),
\end{align*}
which is the first estimate in \eqref{eq:etalam1} since $\eta(t)=D_x u(x,t)$.
\end{proof}

By \cref{prop:estLHsyssol}, \cref{prop:estatRu2},  \cref{lem:estball} and \cref{prop:estatRu1lam},
we obtain the following.

\begin{thm}\label{thm:sHulammain}
Assume {\Hx}, {\Hp} and {\Hc} for $H_0$.
Let $u \in C(\mathbb{R}^n \times [0,T))$ be a viscosity solution of \eqref{eq:geneHJ}--\eqref{eq:ID}.
Then
\begin{enumerate}
\item[(i)]
if $\lambda \neq -C_1$, 
\[
|p| \geqq \underline{I}(x,t;u_0)e^{-(C_1+\lambda)t}-\frac{C_1\beta}{C_1+\lambda}(1-e^{-(C_1+\lambda)t})
\]
for all $(x,t) \in \mathbb{R}^n \times (0,T)$ and $p \in D^-_x u(x,t)$.

\item[(ii)]
if $\lambda=-C_1$, 
\[
|p| \geqq \underline{I}(x,t;u_0)-C_1\beta t
\]
for all $(x,t) \in \mathbb{R}^n \times (0,T)$ and $p \in D^-_x u(x,t)$.
\end{enumerate}
\end{thm}
We also obtain the upper bound for gradients, but we omit it.

\begin{rem}
If $H_0$ satisfies \eqref{eq:1-homo}, then both gradient estimates in \cref{thm:sHulammain} and \eqref{eq:levellb} are true.
Now we compare the lower bound in \eqref{thm:sHulammain} and \eqref{eq:levellb}. 
For $\theta>0$, we define 
\begin{align*}
m(t)&:=\theta e^{-(C_1+\lambda)t} -\beta (e^{-\lambda t}-e^{-(C_1+\lambda)t}), \\
M(t)&:=
\begin{cases}
\theta e^{-(C_1+\lambda)t}-\frac{C_1\beta}{C_1+\lambda}(1-e^{-(C_1+\lambda)t}) & \mbox{if $\lambda \neq -C_1$}, \\
\theta-C_1\beta t & \mbox{if $\lambda=-C_1$}.
\end{cases}
\end{align*}
Then we have $M(0)=m(0)=\theta$ and 
\[
M(t)<m(t) \quad \mbox{if $\lambda \geqq 0$}, \quad 
M(t)>m(t) \quad \mbox{if $\lambda<0$}
\]
for $t \in (0,T]$.
In fact, if $\lambda=-C_1$, then we have
\begin{align*}
m(t)=\theta -\beta (e^{C_1 t}-1), \quad 
M(t)=\theta-C_1\beta t.
\end{align*}
Since $e^{C_1 t}-1>C_1 t$ ($t \in (0,T]$), we have
\[
m(0)=M(0)=\theta, \quad 
m(t)<M(t) \ (t \in (0,T]).
\]
If $\lambda \neq -C_1$, we define 
\begin{align*}
f(t)
&:=M(t)-m(t) \\
&=-\frac{C_1\beta}{C_1+\lambda}(1-e^{-(C_1+\lambda)t}) +\beta (e^{-\lambda t}-e^{-(C_1+\lambda)t}) \\
&=-\frac{\lambda \beta}{C_1+\lambda}e^{-(C_1+\lambda)t}+\beta e^{-\lambda t} -\frac{C_1\beta}{C_1+\lambda}.
\end{align*}
Here $f(0)=0$ and 
\begin{align*}
f^{\prime}(t)
&=\lambda \beta e^{-(C_1+\lambda)t}-\lambda \beta e^{-\lambda t} \\
&=\lambda \beta e^{-\lambda t} (e^{-C_1 t}-1).
\end{align*}
For $t \in (0,T]$, we have
\[
f^{\prime}(t)<0 \quad \mbox{if $\lambda \geqq 0$}, \quad 
f^{\prime}(t)>0 \quad \mbox{if $\lambda<0$ and $\lambda \neq -C_1$}
\]
and therefore 
\[
M(t)<m(t) \quad \mbox{if $\lambda \geqq 0$}, \quad 
M(t)>m(t) \quad \mbox{if $\lambda<0$ and $\lambda \neq -C_1$}.
\]
\end{rem}

\subsection{Concrete Equations}

\begin{ex}[Transport equation]\label{ex:trans}
We begin with the transport equation
\begin{equation}
u_t(x,t)+u(x,t)+\langle x,D_xu(x,t) \rangle=0 
\quad \mathrm{in} \ \mathbb{R}^n \times (0,T).
\label{eq:trans0}
\end{equation}
Here a linear Hamiltonian $H(x,u,p)=u+\langle x,p \rangle$ satisfies {\Hx} with $C_1=1$, $\beta=0$, {\Hp} with $A_2=1$, $B_2=0$, {\Hu} with $K_3=\lambda=1$, and {\Hc}.
By a direct calculation,
the viscosity solution of \eqref{eq:trans0}--\eqref{eq:ID} is given by
\begin{equation}
u(x,t)=e^{-t} u_0(x e^{-t}).
\label{eq:transsol}
\end{equation}
We therefore have
\[
D^-_x u(x,t)
=e^{-2t} D^- u_0 (x e^{-t})
=\{ e^{-2t} p \mid p \in D^- u_0 (x e^{-t}) \}
\]
for all $(x,t) \in \mathbb{R}^n \times (0,T)$ and $R(x,t)=|x|(e^t-1)$.
Since 
\[
|x e^{-t}-x|
=|x|(1-e^{-t})
=e^{-t}R(x,t)
\leqq R(x,t),
\]
we have $x e^{-t} \in \overline{B_{R(x,t)}(x)}$ and obtain
\[
|p| \geqq e^{-2t} \underline{I}(x_0,t;u_0),
\]
which is the lower bound estimate in \cref{thm:leylb}, \cref{thm:ourlb} and \cref{thm:sHulammain}.

Let us consider the initial data
\begin{equation}
u_0(x)=\max\{ 1-|x|, 0 \}. 
\label{eq:exID}
\end{equation}
Then the subgradients of $u_0$ are given as follows:
\begin{equation}
D^- u_0(x)
=\begin{cases}
\{ 0 \} & \mbox{if $|x|>1$}, \\
\left\{ -\dfrac{sx}{|x|} \relmiddle| s \in [0,1] \right\} & \mbox{if $|x|=1$}, \\
\left\{ -\dfrac{x}{|x|} \right\} & \mbox{if $0<|x|<1$}, \\
\emptyset & \mbox{if $x=0$}.
\end{cases}
\label{eq:exa1Dpru0}
\end{equation}
By \eqref{eq:transsol}, we find that 
\[ 
u(x,t)
=\min\{ e^{-t}-e^{-2t}|x|,0 \}.
\]
The graphs of $u_0(x)$ and $u(x,t)$ are shown in \cref{fig:ex1}.

\begin{figure}[htbp]
\captionsetup[subfigure]{font=footnotesize}
\centering
\subcaptionbox{$u_0(x)=\max\{ 1-|x|,0 \}$.}[.42\textwidth]{%
\begin{tikzpicture}[domain=-3:3, samples=200, >=stealth,scale=1.5]
\draw (0,0) node [below left] {O};
\draw[->] (-1.5,0) -- (1.5,0) node[below] {$x$};
\draw[->] (0,-0.2) -- (0,1.5);
\draw[thick] (-1.5,0) -- (-1,0)node[below]{$-1$} -- (0,1)node[right]{1} -- (1,0)node[below]{1} -- (1.5,0);
\end{tikzpicture}
}%
\quad $\longrightarrow$ \quad
\subcaptionbox{$u(x,t)=\max\{ e^{-t}-e^{-2t}|x|,0 \}$.}[.42\textwidth]{%
\begin{tikzpicture}[domain=-3:3, samples=200, >=stealth, scale=1.5]
\draw (0,0) node [below left] {O};
\draw[->] (-1.5,0) -- (1.5,0) node[below] {$x$};
\draw[->] (0,-0.2) -- (0,1.5);
\draw[dotted] (-1.5,0) -- (-1,0) -- (0,1) -- (1,0) -- (1.5,0);
\draw[thick] (-1.5,0) -- (-1.25,0) node[below]{$-e^t$} -- (0,0.75) node[above right]{$e^{-t}$} -- (1.25,0) node[below]{$e^t$} -- (1.5,0);
\end{tikzpicture}
}
\caption{Solution to \eqref{eq:trans0}.}
\label{fig:ex1}
\end{figure}

In this case, we can describe the condition {\Uo} as
\begin{equation}
|p|=1 \quad \mbox{for all $x \in  B_1(0)$ and $p \in D^- u_0(x)$}
\label{eq:exUo}
\end{equation}
and \cref{eq:ourdod} as 
\begin{align*}
\mathcal{E}(0,1)
&=\{ (x,t) \in B_1(0) \times (0,T) \mid |x|+|x|(e^t-1) < 1 \} \\
&=B_{e^{-t}}(0) \times (0,T).   
\end{align*}
By \cref{fig:ex1}(B), we obtain 
\[
|p|=e^{-2t} \quad \mbox{for all $(x,t) \in  B_{e^t}(0) \times (0,T)$ and $p \in D_x^- u(x,t)$},
\]
which means that the estimate \eqref{eq:ourleylb} is optimal, but the domain \eqref{eq:ourdod} is not optimal.

On the other hand, if we consider the equation
\begin{equation}
u_t(x,t)+u(x,t)-\langle x,D_xu(x,t) \rangle=0 
\quad \mathrm{in} \ \mathbb{R}^n \times (0,T),
\label{eq:trans-}
\end{equation}
then the viscosity solution of \eqref{eq:trans-}--\eqref{eq:ID} is given by
\[
u(x,t)=e^{-t} u_0(x e^{t}).
\]
Under the initial data \eqref{eq:exID}, we find that
\[ 
u(x,t)
=\min\{ e^{-t}-|x|,0 \}.
\]
The graphs of $u_0(x)$ and $u(x,t)$ are shown in \cref{fig:ex2}.

\begin{figure}[htbp]
\captionsetup[subfigure]{font=footnotesize}
\centering
\subcaptionbox{$u_0(x)=\max\{ 1-|x|,0 \}$.}[.42\textwidth]{%
\begin{tikzpicture}[domain=-3:3, samples=200, >=stealth,scale=1.5]
\draw (0,0) node [below left] {O};
\draw[->] (-1.5,0) -- (1.5,0) node[below] {$x$};
\draw[->] (0,-0.2) -- (0,1.5);
\draw[thick] (-1.5,0) -- (-1,0)node[below]{$-1$} -- (0,1)node[right]{1} -- (1,0)node[below]{1} -- (1.5,0);
\end{tikzpicture}
}%
\quad $\longrightarrow$ \quad
\subcaptionbox{$u(x,t)=\max\{ e^{-t}-|x|,0 \}$.}[.42\textwidth]{%
\begin{tikzpicture}[domain=-3:3, samples=200, >=stealth, scale=1.5]
\draw (0,0) node [above left] {O};
\draw[->] (-1.5,0) -- (1.5,0) node[below] {$x$};
\draw[->] (0,-0.2) -- (0,1.5);
\draw[dotted] (-1.5,0) -- (-1,0) -- (0,1) -- (1,0) -- (1.5,0);
\draw[thick] (-1.5,0) -- (-0.5,0)node[below]{$-e^{-t}$} -- (0,0.5)node[above right]{$e^{-t}$} -- (0.5,0)node[below]{$e^{-t}$} -- (1.5,0);
\end{tikzpicture}
}
\caption{Solution to \eqref{eq:trans-}.}
\label{fig:ex2}
\end{figure}
By \cref{fig:ex2}(B), we obtain 
\[
|p|=1 \quad \mbox{for all $(x,t) \in  B_{e^{-t}}(0) \times (0,T)$ and $p \in D_x^- u(x,t)$},
\]
which means that the domain \eqref{eq:ourdod} is optimal, but the estimate \eqref{eq:ourleylb} is not optimal.

Furthermore, we consider the equation
\begin{equation}
u_t(x,t)-u(x,t)+\langle x,D_xu(x,t) \rangle=0
\quad \mathrm{in} \ \mathbb{R}^n \times (0,T).
\label{eq:translev}
\end{equation}
This equation is derived from the modified level-set equations (\cite{H.19,BFS.24}).
The viscosity solution of \eqref{eq:translev}--\eqref{eq:ID} is given by
\[
u(x,t)=e^{t} u_0(x e^{-t}).
\]
Under the initial data \eqref{eq:exID}, we find that
\[ 
u(x,t)
=\min\{ e^{t}-|x|,0 \}.
\]
The graphs of $u_0(x)$ and $u(x,t)$ are shown in \cref{fig:ex3}.

\begin{figure}[htbp]
\captionsetup[subfigure]{font=footnotesize}
\centering
\subcaptionbox{$u_0(x)=\max\{ 1-|x|,0 \}$.}[.42\textwidth]{%
\begin{tikzpicture}[domain=-3:3, samples=200, >=stealth,scale=1.5]
\draw (0,0) node [below left] {O};
\draw[->] (-1.5,0) -- (1.5,0) node[below] {$x$};
\draw[->] (0,-0.2) -- (0,1.5);
\draw[thick] (-1.5,0) -- (-1,0)node[below]{$-1$} -- (0,1)node[right]{1} -- (1,0)node[below]{1} -- (1.5,0);
\end{tikzpicture}
}%
\quad $\longrightarrow$ \quad
\subcaptionbox{$u(x,t)=\max\{ e^{t}-|x|,0 \}$.}[.42\textwidth]{%
\begin{tikzpicture}[domain=-3:3, samples=200, >=stealth, scale=1.5]
\draw (0,0) node [below left] {O};
\draw[->] (-1.5,0) -- (1.5,0) node[below] {$x$};
\draw[->] (0,-0.2) -- (0,1.5);
\draw[dotted] (-1.5,0) -- (-1,0) -- (0,1) -- (1,0) -- (1.5,0);
\draw[thick] (-1.5,0) -- (-1.25,0)node[below]{$-e^{t}$} -- (0,1.25)node[above right]{$e^{t}$} -- (1.25,0)node[below]{$e^{t}$} -- (1.5,0);
\end{tikzpicture}
}
\caption{Solution to \eqref{eq:translev}.}
\label{fig:ex3}
\end{figure}
By \cref{fig:ex3}(B), we obtain 
\[
|p|=1 \quad \mbox{for all $(x,t) \in  B_{e^{t}}(0) \times (0,T)$ and $p \in D_x^- u(x,t)$},
\]
which means that both \eqref{eq:ourdod} and \eqref{eq:ourleylb} are not optimal.
However, we can obtain the optimal gradient estimate by \cref{thm:sHulammain}(ii).
\end{ex}

\begin{ex}[Eikonal equation]\label{ex:eiko}
We next consider the eikonal type equation
\begin{equation}
u_t(x,t)+u(x,t)+c|D_xu(x,t)|=0 
\quad \mathrm{in} \ \mathbb{R}^n \times (0,T).
\label{eq:eiko}
\end{equation}
Here a Hamiltonian $H(u,p)=u+c|p|$ ($c>0$) satisfies {\Hx} with $C_1=0$, $\beta=0$, {\Hp} with $A_2=0$, $B_2=c$,
{\Hu} with $K_3=\lambda=1$, and {\Hc}.
By the Hopf--Lax formula,
the viscosity solution of \eqref{eq:eiko}--\eqref{eq:ID} is given by
\begin{equation}
u(x,t)
=\min_{y \in \overline{B_{ct} (x)}} e^{-t} u_0(y).
\label{eq:eikosol}
\end{equation}
Moreover, we have $R(x,t)=ct$.

Now let us consider the initial data as \eqref{eq:exID}.
By \eqref{eq:eikosol}, we find that 
\[ 
u(x,t)
=\max\{ e^{-t}(1-ct-|x|),0 \}.
\]
The graphs of $u_0(x)$ and $u(x,t)$ are shown in \cref{fig:ex3}.
\begin{figure}[htbp]
\captionsetup[subfigure]{font=footnotesize}
\centering
\subcaptionbox{$u_0(x)=\max\{ 1-|x|,0 \}$.}[.42\textwidth]{%
\begin{tikzpicture}[domain=-3:3, samples=200, >=stealth,scale=1.5]
\draw (0,0) node [below left] {O};
\draw[->] (-1.5,0) -- (1.5,0) node[below] {$x$};
\draw[->] (0,-0.2) -- (0,1.5);
\draw[thick] (-1.5,0) -- (-1,0)node[below]{$-1$} -- (0,1)node[right]{1} -- (1,0)node[below]{1} -- (1.5,0);
\end{tikzpicture}
}%
\quad $\longrightarrow$ \quad
\subcaptionbox{$u(x,t)=\max\{ e^{-t}(1-ct-|x|),0 \}$.}[.42\textwidth]{%
\begin{tikzpicture}[domain=-3:3, samples=200, >=stealth, scale=1.5]
\draw (0,0) node [below left] {O};
\draw[->] (-1.5,0) -- (1.5,0) node[below] {$x$};
\draw[->] (0,-0.2) -- (0,1.5);
\draw[dotted] (-1.5,0) -- (-1,0) -- (0,1) -- (1,0) -- (1.5,0);
\draw[thick] (-1.5,0) -- (-0.5,0)node[below left]{$-1+ct$} -- (0,0.3)node[above right]{$e^{-t}(1-ct)$} -- (0.5,0)node[below]{$1-ct$} -- (1.5,0);
\end{tikzpicture}
}
\caption{Solution to \eqref{eq:eiko}.}
\label{fig:ex4}
\end{figure}

In this case, we can describe \eqref{eq:ourdod} as 
\begin{align*}
\mathcal{E}(0,1)
&=\{ (x,t) \in B_1(0) \times (0,T) \mid |x|+ct < 1 \} \\
&=B_{1-ct}(0) \times (0,T). 
\end{align*}
By \cref{fig:ex4}(B), we obtain 
\[
|p|=e^{-t} \quad \mbox{for all $(x,t) \in  B_{1-ct}(0) \times (0,T)$ and $p \in D_x^- u(x,t)$},
\]
which means that both \eqref{eq:ourleylb} and \eqref{eq:ourdod} are optimal.
\end{ex}

\appendix
\section{Local comparison principle}\label{sec:LCP}

In the Appendix, we prove some results used in the main proofs.
First, we establish the local comparison principle
(see \cite[Remark 2.4]{HH.23}, \cite[Theorem 6.1]{L.01}, \cite[Theorem V.3]{CL.83}, \cite[Theorem 2.4]{I.84}, \cite[Chapter 2, Theorem 5.3]{ABIL.11}).
In the proof of \cref{thm:leylb}, we analyze the difference between the subsolution $u_{\varepsilon}$ and the supersolution $u$.
Since $u_{\varepsilon} \leqq u$ in $\mathbb{R}^n \times (0, T)$, we may assume that a subsolution is less than or equal to a supersolution.

\begin{thm}[Local comparison principle]\label{thm:unCP}
Assume {\Hx}--{\Hu}.
Let $x_0 \in \mathbb{R}^n$, $r>0$ and $f,g \in C(\mathbb{R}^n \times [0,T))$.
Let $u \in C(\mathbb{R}^n \times (0,T))$ be a viscosity subsolution of 
\begin{equation}
u_t(x,t)+H(x,t,u(x,t),D_x u(x,t))=f(x,t)
\quad \mathrm{in} \ B_r(x_0) \times (0,T),
\label{eq:unCPsub}
\end{equation}
and $v \in C(\mathbb{R}^n \times (0,T))$ be a viscosity supersolution of 
\begin{equation}
v_t(x,t)+H(x,t,v(x,t),D_x v(x,t))=g(x,t)
\quad \mathrm{in} \ B_r(x_0) \times (0,T).
\label{eq:unCPsuper}
\end{equation}
Assume $u \leqq v$ in $\mathbb{R}^n \times (0, T)$.
Then we have 
\begin{equation}
\begin{split}
(u-v)(x,t)
&\leqq \sup_{y \in \overline{B_r(x_0)}} e^{-K_3t}(u-v)(y,0) \\
&\hspace{1cm} +e^{-K_3t}\int_0^t e^{K_3s} \sup_{y \in \overline{B_r(x_0)}} (f-g)(y,s)\,ds  
\end{split}
\label{eq:unCP}
\end{equation}
for all $(x,t) \in \overline{\mathcal{E}(x_0,r)}$,
where $\mathcal{E}(x_0,r)$ is defined in \eqref{eq:ourdod}.
\end{thm}

To prove this theorem,
we use the fundamental result.

\begin{lem}\label{lem:unCPlem}
Under the settings in \cref{thm:unCP},
$w:=u-v$ is a viscosity subsolution of 
\begin{equation}
\begin{split}
w_t(x,t)+K_3w(x,t)-(A_2|x|+B_2)|D_x w(x,t)|=(f-g)(x,t) \\
\quad \mathrm{in} \ B_r(x_0) \times (0,T).  
\end{split}
\label{eq:unCPwsub}
\end{equation}
\end{lem}
\begin{proof}
Let $\phi \in C^1(B_r(x_0) \times (0,T))$ such that
$w-\phi$ has a strict local maximum at $(\hat{x},\hat{t}) \in B_r(x_0) \times (0,T)$.
We may assume that
$w-\phi$ has a strict maximum at $(\hat{x},\hat{t})$ over $\overline{B_{\tau}(\hat{x})} \times [\hat{t}-\tau,\hat{t}+\tau]$ ($\tau>0$).
For $\varepsilon,\alpha>0$,
we define $\Phi_{\varepsilon,\alpha}: \overline{B_{\tau}(\hat{x})} \times \overline{B_{\tau}(\hat{x})} \times [\hat{t}-\tau,\hat{t}+\tau] \times [\hat{t}-\tau,\hat{t}+\tau] \to \mathbb{R}$ as
\[
\Phi_{\varepsilon,\alpha}(x,t,y,s):=
u(x,t)-v(y,s)-\dfrac{|x-y|^2}{\varepsilon^2}-\dfrac{|t-s|^2}{\alpha^2}-\phi(x,t).
\]
Let $(x_{\varepsilon,\alpha},y_{\varepsilon,\alpha},t_{\varepsilon,\alpha},s_{\varepsilon,\alpha}) \in (\overline{B_{\tau}(\hat{x})} \times [\hat{t}-\tau,\hat{t}+\tau])^2$
be a maximum point of $\Psi_{\varepsilon,\alpha}$.
By the classical argument (for example, see \cite{CIL.92}),
we have
$(x_{\varepsilon,\alpha},y_{\varepsilon,\alpha},t_{\varepsilon,\alpha},s_{\varepsilon,\alpha}) \in (B_r(x_0) \times (0,T))^2$ for small enough $\varepsilon,\alpha>0$
and 
there exist $x_{\varepsilon},y_{\varepsilon} \in B_r(x_0)$ such that
$(x_{\varepsilon,\alpha},y_{\varepsilon,\alpha}) \to (x_{\varepsilon},y_{\varepsilon})$ and 
$t_{\varepsilon,\alpha},s_{\varepsilon,\alpha} \to t_{\varepsilon}$ 
as $\alpha \to +0$.

Since $u$ is a viscosity subsolution of \eqref{eq:unCPsub},
we have
\begin{align*}
&\dfrac{2(t_{\varepsilon,\alpha}-s_{\varepsilon,\alpha})}{\alpha^2}
 +\phi_t(x_{\varepsilon,\alpha},t_{\varepsilon,\alpha}) \\
&\hspace{1cm} 
+H\left(x_{\varepsilon,\alpha},t_{\varepsilon,\alpha},u(x_{\varepsilon,\alpha},t_{\varepsilon,\alpha}),\frac{2(x_{\varepsilon,\alpha}-y_{\varepsilon,\alpha})}{\varepsilon^2}+D_x\phi(x_{\varepsilon,\alpha},t_{\varepsilon,\alpha}) \right) \\
&\hspace{1cm}
 \leqq f(x_{\varepsilon,\alpha},t_{\varepsilon,\alpha})
\end{align*}
and since $v$ is a viscosity supersolution of \eqref{eq:unCPsuper},
we have
\[
\dfrac{2(t_{\varepsilon,\alpha}-s_{\varepsilon,\alpha})}{\alpha^2}
 +H \left(y_{\varepsilon,\alpha},s_{\varepsilon,\alpha},v(y_{\varepsilon,\alpha},s_{\varepsilon,\alpha}),\frac{2(x_{\varepsilon,\alpha}-y_{\varepsilon,\alpha})}{\varepsilon^2} \right) 
\geqq g(y_{\varepsilon,\alpha},s_{\varepsilon,\alpha}).
\]
Combining them,
we obtain
\begin{align*}
& \phi_t(x_{\varepsilon,\alpha},t_{\varepsilon,\alpha})
 +H \left(x_{\varepsilon,\alpha},t_{\varepsilon,\alpha},u(x_{\varepsilon,\alpha},t_{\varepsilon,\alpha}),\frac{2(x_{\varepsilon,\alpha}-y_{\varepsilon,\alpha})}{\varepsilon^2}+D_x\phi(x_{\varepsilon,\alpha},t_{\varepsilon,\alpha}) \right) \\
& \hspace{0.5cm} -H\left(y_{\varepsilon,\alpha},s_{\varepsilon,\alpha},v(y_{\varepsilon,\alpha},s_{\varepsilon,\alpha}),\frac{2(x_{\varepsilon,\alpha}-y_{\varepsilon,\alpha})}{\varepsilon^2} \right) \\
&\leqq f(x_{\varepsilon,\alpha},t_{\varepsilon,\alpha})-g(y_{\varepsilon,\alpha},s_{\varepsilon,\alpha} ).
\end{align*}
Letting $\alpha \to +0$, we have
\begin{align*}
& \phi_t(x_{\varepsilon},t_{\varepsilon})
 +H \left(x_{\varepsilon},t_{\varepsilon},u(x_{\varepsilon},t_{\varepsilon}),\frac{2(x_{\varepsilon}-y_{\varepsilon})}{\varepsilon^2}+D_x\phi(x_{\varepsilon},t_{\varepsilon}) \right) \\
& \hspace{1cm} -H \left(y_{\varepsilon},t_{\varepsilon},v(y_{\varepsilon},t_{\varepsilon}),\frac{2(x_{\varepsilon}-y_{\varepsilon})}{\varepsilon^2} \right)
\leqq f(x_{\varepsilon},t_{\varepsilon})-g(y_{\varepsilon},t_{\varepsilon}).
\end{align*}
By {\Hx}--{\Hu}, we obtain
\begin{align*}
& \phi_t(x_{\varepsilon},t_{\varepsilon}) -K_3|u(x_{\varepsilon},t_{\varepsilon})-v(y_{\varepsilon},t_{\varepsilon})| -(A_2|x_{\varepsilon}|+B_2)|D_x\phi(x_{\varepsilon},t_{\varepsilon})| \\
&\leqq f(x_{\varepsilon},t_{\varepsilon})-g(y_{\varepsilon},t_{\varepsilon})
  +C_1\beta|x_{\varepsilon}-y_{\varepsilon}|+2C_1\frac{|x_{\varepsilon}-y_{\varepsilon}|^2}{\varepsilon^2}.
\end{align*}
Letting $\varepsilon \to +0$, we have
\[
\phi_t(\hat{x},\hat{t}) -K_3|w(\hat{x},\hat{t})| -(A_2|\hat{x}|+B_2)|D_x\phi(\hat{x},\hat{t})| 
\leqq (f-g)(\hat{x},\hat{t}),
\]
which implies that 
$w$ is a viscosity subsolution of \eqref{eq:unCPwsub}
since $w \leqq 0$.
\end{proof}

\begin{proof}[Proof of \cref{thm:unCP}]
By \cref{lem:unCPlem},
$w=u-v$ is a viscosity subsolution of \eqref{eq:unCPwsub}.
Moreover,
$W(x,t):=e^{K_3t}w(x,t)$ is a viscosity subsolution of 
\begin{equation}
\begin{split}
W_t(x,t)-(A_2|x|+B_2)|D_x W(x,t)|=e^{K_3t}(f-g)(x,t) \\
\quad \mathrm{in} \ B_r(x_0) \times (0,T)  
\end{split}
\label{eq:unCPWsub}
\end{equation}
(for example, see \cite[Proposition II.2.5]{BC.97}).
For $\eta>0$ and $0<\varepsilon<r$,
we define 
$\Psi_{\varepsilon,\eta}: \overline{B_r(x_0)} \times [0,T) \to \mathbb{R}$ as
\begin{align*}
\Psi_{\varepsilon,\eta}(x,t)
&:=W(x,t)-\eta t-\dfrac{\eta}{T-t} \\
&\hspace{1cm} -\int_0^t e^{K_3s}\sup_{y \in \overline{B_r(x_0)}} (f-g)(y,s)\,ds -\chi_{\varepsilon}(h_{\varepsilon}(x,t)), 
\end{align*}
where 
$h_{\varepsilon}(x,t) \in C^1(B_r(x_0) \times (0,T))$ is a nonnegative function and 
$\chi_{\varepsilon}: [0,\infty) \to [0,\infty]$ satisfies that
$\chi_{\varepsilon} \in C^1(0,r)$ is increasing in $[0,r)$ and 
\[
\chi_{\varepsilon} \equiv 0 \ \mathrm{in} \ [0,r-\varepsilon], \quad
\lim_{z \to r} \chi_{\varepsilon}(z)=\infty, \quad
\chi_{\varepsilon} \equiv \infty \ \mathrm{in} \ [r,\infty).
\]
According to the classical argument,
$\Psi_{\varepsilon,\eta}$ attains a maximum at $(\hat{x},\hat{t}) \in B_r(x_0) \times [0,T)$
over $\overline{B_r(x_0)} \times [0,T)$.

Suppose $\hat{t}>0$.
Since $W$ is a viscosity subsolution of \eqref{eq:unCPWsub},
we have 
\begin{align*}
\eta+\frac{\eta}{(T-\hat{t})^2}
&+e^{K_3 \hat{t}} 
 \left\{ \sup_{y \in \overline{B_r(x_0)}}  (f-g)(y,\hat{t}) -(f-g)(\hat{x},\hat{t}) \right\} \\
&+\chi_{\varepsilon}'(h_{\varepsilon}(x,t)) \left\{ 
 (h_{\varepsilon})_t(\hat{x},\hat{t})-(A_2|\hat{x}|+B_2)|D_xh_{\varepsilon}(\hat{x},\hat{t})|
 \right\}
\leqq 0,
\end{align*}
which is a contradiction if $h_{\varepsilon}$ satisfies
\begin{equation}
(h_{\varepsilon})_t(\hat{x},\hat{t})-(A_2|\hat{x}|+B_2)|D_xh_{\varepsilon}(\hat{x},\hat{t})| \geqq 0.
\label{eq:CPhineq}
\end{equation}

We now define 
\begin{equation}
h_{\varepsilon}(x,t):=
\begin{cases}
B_2t +\sqrt{|x-x_0|^2+\varepsilon^2} & (A_2=0), \\ \smallskip
\left( \dfrac{B_2}{A_2}+\sqrt{|x|^2+\varepsilon^2} \right)(e^{A_2t}-1) +\sqrt{|x-x_0|^2+\varepsilon^2} & (A_2 \neq 0).
\end{cases} 
\label{eq:CPhdef}
\end{equation}
Then, we find that $h_{\varepsilon} \in C^1(B_r(x_0) \times (0,T))$, $h_{\varepsilon} \geqq 0$ and 
\cref{eq:CPhineq} holds by \cref{rem:cal}.
Therefore, we have $\hat{t}=0$ and obtain
\begin{align*}
\Psi_{\varepsilon}(x,t)
&\leqq \Psi_{\varepsilon}(\hat{x},0)
=W(\hat{x},0)-\frac{\eta}{T}-\chi_{\varepsilon}(h_{\varepsilon}(\hat{x},0)) \\
&\leqq W(\hat{x},0)
=(u-v) (\hat{x},0) \\
&\leqq \sup_{y \in \overline{B_r(x_0)}} (u-v)(y,0),
\end{align*}
that is,
\begin{equation}
\begin{split}
&W(x,t)-\eta t-\dfrac{\eta}{T-t} \\
&\leqq 
\sup_{y \in \overline{B_r(x_0)}} (u-v)(y,0)
 +\int_0^t e^{K_3s}\sup_{y \in \overline{B_r(x_0)}} (f-g)(y,s)\,ds
  +\chi_{\varepsilon}(h_{\varepsilon}(x,t))
\end{split}
\label{eq:unCPdom} 
\end{equation}
for all $(x,t) \in B_r(x_0) \times (0,T)$.
If $h_{\varepsilon}(x,t) \leqq r-\varepsilon$,
we have $\chi_{\varepsilon}(h_{\varepsilon}(x,t))=0$.
Letting $\varepsilon \to 0$ for $h_{\varepsilon}(x,t) \leqq r-\varepsilon$,
we obtain
\begin{equation}
R(x,t)+|x-x_0| \leqq r,
\label{eq:defReq}
\end{equation}
where $R(x,t)$ is defined in \eqref{eq:defiR}.
We remark that $(x,t)$ which satisfies \eqref{eq:defReq} belongs to $\mathcal{E}(x_0,r)$.
Letting $\eta \to 0$ in \eqref{eq:unCPdom},
we have 
\begin{align*}
e^{K_3t}(u-v)(x,t)
\leqq \sup_{y \in \overline{B_r(x_0)}} (u-v)(y,0)
 +\int_0^t e^{K_3s}\sup_{y \in \overline{B_r(x_0)}} (f-g)(y,s)\,ds
\end{align*}
for all $(x,t) \in \mathcal{E}(x_0,r)$, which completes the proof.
\end{proof}

\begin{rem}
We only use the assumption $u \leqq v$ in $\mathbb{R}^n \times (0,T)$ to obtain \eqref{eq:unCPWsub} from \eqref{eq:unCPwsub} since this is true for $w \leqq 0$.
This assumption holds under further assumptions, such as uniform continuity and monotonicity on the variable $u$, because we can prove the comparison principle under these conditions.
\end{rem}

\begin{rem}\label{rem:cal}
Let us confirm that $h_{\varepsilon}$ defined in \eqref{eq:CPhdef} satisfies \eqref{eq:CPhineq}.
If $A_2=0$, then 
\begin{align*}
(h_{\varepsilon})_t(x,t)=B_2, \quad
D_xh_{\varepsilon}(x,t)=\frac{x-x_0}{\sqrt{|x-x_0|^2+\varepsilon^2}},
\end{align*}
and therefore 
\begin{align*}
(h_{\varepsilon})_t(x,t)-(A_2|x|+B_2)|D_xh_{\varepsilon}(x,t)|
&=B_2-B_2\frac{|x-x_0|}{\sqrt{|x-x_0|^2+\varepsilon^2}} \\
&\geqq B_2-B_2
=0.
\end{align*}
If $A_2 \neq 0$, then
\begin{align*}
(h_{\varepsilon})_t(x,t)&=A_2e^{A_2t} \left( \frac{B_2}{A_2}+\sqrt{|x|^2+\varepsilon^2} \right), \\
D_xh_{\varepsilon}(x,t)&=\frac{x}{\sqrt{|x|^2+\varepsilon^2}} (e^{A_2t}-1)+ \frac{x-x_0}{\sqrt{|x-x_0|^2+\varepsilon^2}},
\end{align*}
and therefore 
\begin{align*}
&(h_{\varepsilon})_t(x,t)-(A_2|x|+B_2)|D_xh_{\varepsilon}(x,t)| \\
&\geqq A_2e^{A_2t} \left( \frac{B_2}{A_2}+\sqrt{|x|^2+\varepsilon^2} \right) \\
&\hspace{0.5cm}  
 -(A_2|x|+B_2) \left( \frac{|x|}{\sqrt{|x|^2+\varepsilon^2}} (e^{A_2t}-1) +(A_2|x|+B_2)\frac{|x-x_0|}{\sqrt{|x-x_0|^2+\varepsilon^2}} \right) \\
&\geqq A_2e^{A_2t} \left( \frac{B_2}{A_2}+|x| \right)
 -e^{A_2t}(A_2|x|+B_2)
 =0.
\end{align*}
Therefore $h_{\varepsilon}$ satisfies \eqref{eq:CPhineq}.
\end{rem}

\begin{rem}
In \cite[Proof of Theorem 6.1]{L.01},
$h_{\varepsilon}$ is defined as 
\[
h_{\varepsilon}(x,t)
=e^{(A_2+B_2+A_2|x_0|)t}(1+\sqrt{|x-x_0|^2+\varepsilon^2})-1
\]
and \cref{eq:unCP} holds in $\mathcal{D}(x_0,r)$ for $K_3=0$,
where $\mathcal{D}(x_0,r)$ is defined in \eqref{eq:leydod}.
\end{rem}

\section{Equvalence to two solutions}\label{sec:equiv}

\begin{proof}[Proof of \cref{thm:equivvb}]
Assume that $u$ is a Barron--Jensen solution of \eqref{eq:HJu}--\eqref{eq:ID}.
It is obvious that $u$ is a viscosity supersolution of \eqref{eq:HJu}--\eqref{eq:ID} by \eqref{rem:vis}.
Moreover, by \cref{lem:infconvsol}, 
$u_{\varepsilon,\alpha}$ defined in \eqref{eq:infconvt} is a viscosity subsolution of \eqref{eq:infconvapeq} for $\varepsilon,\alpha$ small enough.
Letting $\varepsilon,\alpha \to 0$,
a discontinuous stability theorem (\cite[Section 6]{CIL.92}) implies that
$u$ is a viscosity subsolution of \eqref{eq:HJu}--\eqref{eq:ID} in $\mathcal{A}_{\rho}$.
Since $\bigcup \mathcal{A}_{\rho}=\mathbb{R}^n \times (0,T)$,
we have $u$ is a viscosity solution of \eqref{eq:HJu}--\eqref{eq:ID} in $\mathbb{R}^n \times (0,T)$.

Assume that $u$ is a viscosity solution of \eqref{eq:HJu}--\eqref{eq:ID}.
Since $u$ is a viscosity subsolution of \eqref{eq:HJu}--\eqref{eq:ID},
the sup-convolution
\[
u^{\varepsilon,\alpha}(x,t)
:=\sup_{(y,s) \in \overline{\mathcal{A}}}
 \left\{  u(y,s)-e^{\gamma t}\frac{|x-y|^2}{\varepsilon^2}-\frac{|t-s|^2}{\alpha^2} \right\}  
\]
is a viscosity subsolution of
\[
\begin{split}
&u_t(x,t)+H(x,t,u(x,t),D_xu(x,t)) \\
&=\dfrac{C_1\beta}{2}e^{\gamma t}\varepsilon^2 +K_3\omega_{\mathcal{A}}(\mathcal{M}_{\mathcal{A}}(\varepsilon+\alpha)) +\mu_{\varepsilon,\mathcal{A}}(\mathcal{M}_{\mathcal{A}}\alpha)
\quad \mathrm{in} \ \mathcal{A}_{\rho}.
\end{split}
\]
By \cref{lem:aesol},
$u^{\varepsilon,\alpha}$ is also a viscosity supersolution of 
\[
\begin{split}
&-u_t(x,t)-H(x,t,u(x,t),D_xu(x,t)) \\
&=-\dfrac{C_1\beta}{2}e^{\gamma t}\varepsilon^2 -K_3\omega_{\mathcal{A}}(\mathcal{M}_{\mathcal{A}}(\varepsilon+\alpha)) -\mu_{\varepsilon,\mathcal{A}}(\mathcal{M}_{\mathcal{A}}\alpha)
\quad \mathrm{in} \ \mathcal{A}_{\rho}.
\end{split}
\]
Letting $\varepsilon,\alpha \to 0$,
a discontinuous stability theorem implies that
$u$ is a viscosity subsolution of 
\begin{equation}
-u_t(x,t)-H(x,t,u(x,t),D_x u(x,t))=0
\quad \mathrm{in} \ \mathcal{A}_{\rho}. 
\label{eq:equivB1}
\end{equation}
Since $\bigcup \mathcal{A}_{\rho}=\mathbb{R}^n \times (0,T)$,
$u$ is a viscosity subsolution of \eqref{eq:equivB1} in $\mathbb{R}^n \times (0,T)$.
Combining this fact and the fact that $u$ is a viscosity supersolution of \eqref{eq:HJu}--\eqref{eq:ID},
We deduce that $u$ is a Barron--Jensen solution of \eqref{eq:HJu}--\eqref{eq:ID}.
\end{proof}

\section*{Acknowledgments}
The author would like to thank Professor Nao Hamamuki for the fruitful discussions and insightful suggestions.
This work was partially supported by JST SPRING, Grant Number JPMJSP2119, and JSPS KAKENHI, Grant Number 24KJ0269.

\bibliographystyle{plain}
\bibliography{references}
\end{document}